\documentclass[a4paper,12pt]{amsart}
\usepackage{mathrsfs}
\usepackage{appendix}
\usepackage{amssymb}
\usepackage{fancyhdr}
\usepackage{charter}
\usepackage{typearea}
\usepackage{color}
\usepackage{amsmath,amstext,amsthm,amscd}

\usepackage[a4paper,top=3cm,bottom=3cm,left=3cm,right=3cm]{geometry}

\makeatletter
      \def\@setcopyright{}
      \def\serieslogo@{}
      \makeatother

\newcommand{\Complex}{\mathbb C}
\newcommand{\Real}{\mathbb R}
\newcommand{\N}{\mathbb N}
\newcommand{\ddbar}{\overline\partial}
\newcommand{\pr}{\partial}
\newcommand{\ol}{\overline}
\newcommand{\Td}{\widetilde}
\newcommand{\norm}[1]{\left\Vert#1\right\Vert}
\newcommand{\abs}[1]{\left\vert#1\right\vert}

\newcommand{\set}[1]{\left\{#1\right\}}
\newcommand{\To}{\rightarrow}

\newcommand{\R}{\mathbb{R}}
\newcommand{\C}{\mathbb{C}}

\theoremstyle{plain}
\newtheorem{theorem}{Theorem}[section]
\newtheorem{lemma}[theorem]{Lemma}
\newtheorem{corollary}[theorem]{Corollary}

\newtheorem{definition}[theorem]{Definition}

\newtheorem{remark}[theorem]{Remark}
\numberwithin{equation}{section}

\begin{document}
\title[{An explicit formula for Szeg\H{o} kernels on the Heisenberg group}]
{An explicit formula for Szeg\H{o} kernels on the Heisenberg group}
\author[Hendrik Herrmann]{Hendrik Herrmann}
\address{Mathematical Institute, University of Cologne, Weyertal 86-90, 50931 Cologne, Germany}
\thanks{Hendrik Herrmann was partially supported by the CRC TRR 191: ``Symplectic Structures in Geometry, Algebra and Dynamics''. He would like to thank  the Mathematical Institute, Academia Sinica, and the School of Mathematics and Statistics, Wuhan University, for hospitality, a comfortable accommodation and financial support during his visits in January and March - April, respectively.}
\email{heherrma@math.uni-koeln.de or post@hendrik-herrmann.de}
\author[Chin-Yu Hsiao]{Chin-Yu Hsiao}
\address{Institute of Mathematics, Academia Sinica and National Center for Theoretical Sciences, Astronomy-Mathematics Building, No. 1, Sec. 4, Roosevelt Road, Taipei 10617, Taiwan}
\thanks{Chin-Yu Hsiao was partially supported by Taiwan Ministry of Science of Technology project 104-2628-M-001-003-MY2 and the Golden-Jade fellowship of Kenda Foundation}
\email{chsiao@math.sinica.edu.tw or chinyu.hsiao@gmail.com}
\author[Xiaoshan Li]{Xiaoshan Li}
\address{School of Mathematics
and Statistics, Wuhan University, Hubei 430072, China}
\thanks{Xiaoshan Li was supported by  National Natural Science Foundation of China (Grant No. 11501422).}
\email{xiaoshanli@whu.edu.cn}
\dedicatory{In memory of Professor Qikeng Lu}

\begin{abstract}
In this paper, we give an explicit formula for the Szeg\H{o} kernel for $(0,q)$ forms on the Heisenberg group $H_{n+1}$.
\end{abstract}

\maketitle \tableofcontents

\section{Introduction}

Let $(X, T^{1,0}X)$ be a CR manifold of dimension $2n+1$, $n\geq1$, and let $\Box^{(q)}_b$ be the Kohn Lalpacian acting on $(0,q)$ forms.
The orthogonal projection $S^{(q)}:L^2_{(0,q)}(X)\To {\rm Ker\,}\Box^{(q)}_b$ onto ${\rm Ker\,}\Box^{(q)}_b$
is called the Szeg\H{o} projection, while its distribution kernel $S^{(q)}(x,y)$ is called the Szeg\H{o} kernel.
The study of the Szeg\H{o} projection and kernel is a classical and important subject in several complex variables and CR geometry.
When $X$ is compact, strongly pseudoconvex and $\Box^{(0)}_b$ has $L^2$ closed range,
Boutet de Monvel-Sj\"ostrand~\cite{BM76} showed that $S^{(0)}(x,y)$
is a complex Fourier integral operator with complex phase. In particular,
$S^{(0)}(x,y)$ is smooth outside the diagonal of $X\times X$
and there is a precise description of the singularity on the diagonal $x=y$,
where $S^{(0)}(x,x)$ has a certain asymptotic expansion.
The second-named author ~\cite{Hsiao08} showed that if $X$ is compact, the Levi form is non-degenerate and
$\Box^{(q)}_b$ has
$L^2$ closed range for some $q\in\set{0,1,\ldots,n-1}$, then $S^{(q)}(x,y)$ is a complex Fourier integral operator.

When $X$ is non-compact or $\Box^{(q)}_b$ has no $L^2$ closed range, it is very difficult to study the Szeg\H{o} kernel. In this work, we give an explicit formula for the Szeg\H{o} kernel for $(0,q)$ forms on the Heisenberg group $H_{n+1}=\mathbb C^n\times \mathbb R$. Our results tell us that in the Heisenberg group case, the Szeg\H{o} kernel for $(0,q)$ forms  is also a complex Fourier integral operator. Note that in the Heisenberg group case, $\Box^{(q)}_b$ may has no $L^2$ closed range.

Only few examples of CR manifolds with explicit Szeg\H{o} kernels are known. To give a closed formula of the Szeg\H{o} kernel is not only a problem of its own interest but also significant for the general theory. For example, when $X$ is asymptotically flat, the explicit Szeg\H{o} kernel on the Heisenberg group was used in positive mass theorem in CR geometry~\cite{CMY17}, ~\cite{HY13},~\cite{HY15}.

We now formulate our main results. We refer to Section~\ref{s:prelim} for some notations and terminology used here. Let $H_{n+1}=\mathbb C^{n}\times\mathbb R$ be the Heisenberg group. We use $x=(z, x_{2n+1})=(x_1,\ldots,x_{2n+1})$ to denote the  coordinates on $H_{n+1}$, where $z=(z_1,\ldots,z_n)$, $z_j=x_{2j-1}+ix_{2j}$, $j=1,\ldots,n$. We denote by $T^{1, 0}H_{n+1}$ the CR structure on $H_{n+1}$ which is given by $$T^{1, 0}H_{n+1}:={\rm span}_{\mathbb C}\{Z_j: Z_j=\frac{\partial}{\partial z_j}-i\lambda_j\overline z_j\frac{\partial}{\partial x_{2n+1}}, j=1, \cdots, n\}$$
where $\lambda_j\in\mathbb R, \forall j$, are given real numbers.  We denote by $T^{0, 1}H_{n+1}$ the complex conjugate of $T^{1, 0}H_{n+1}$. Set $T=-\frac{\partial}{\partial x_{2n+1}}$.
Fix a Hermitian metric on $\Complex TH_{n+1}$ denoted by $\langle\cdot|\cdot\rangle$ such that
\begin{equation*}
\begin{split}
&T\bot T^{1, 0}H_{n+1}\bot T^{0, 1}H_{n+1},\\
&\langle T|Z_j\rangle=0, \langle Z_j|Z_k\rangle =\langle\ol Z_j|\ol Z_k\rangle=\delta_{jk}, \forall j, k=1, \cdots, n.
\end{split}
\end{equation*}
Take $d\mu_{H_{n+1}}:=2^{n}dx_1\wedge\cdots\wedge dx_{2n+1}$ be the volume form on $H_{n+1}$.

Denote by $T^{\ast 1,0}H_{n+1}$ and $T^{\ast0,1}H_n$ the dual bundles of
$T^{1,0}H_{n+1}$ and $T^{0,1}H_{n+1}$, respectively. Define the vector bundle of $(0,q)$-forms by
$\Lambda^qT^{\ast0,1}H_{n+1}$. Put
\[\Lambda^{\bullet}T^{\ast0,1}H_{n+1}:=\oplus_{q=0}^n\Lambda^qT^{\ast0,1}H_{n+1}.\]
The Hermitian metric $\langle\cdot|\cdot\rangle$ on
$\mathbb CTH_{n+1}$ induces by duality a Hermitian metric on $\mathbb CT^\ast H_{n+1}$ and also on $\Lambda^{\bullet}T^{\ast0,1}H_{n+1}$. We shall also denote all these induced
metrics by $\langle\cdot|\cdot\rangle$. We can check that the dual frame of $\{Z_j, \overline Z_j, -T;\, j=1,\ldots,n\}_{j=1}^n$ is $\{dz_j, d\overline z_j, \omega_0;\, j=1,2,\ldots,n\}$,  where $dz_j=dx_{2j-1}+idx_{2j}$, $j=1,\ldots,n$, and
\begin{equation}\label{e-gue170522}
\omega_0(x)=dx_{2n+1}+\sum\limits_{j=1}^ni(\lambda_j\overline z_jdz_j-\lambda_jz_jd\overline z_j).
\end{equation}
Thus, one has
\[\Lambda^qT^{\ast 0,1}H_{n+1}={\rm span}_{\mathbb C}\{d\ol z_{j_1}\wedge\cdots\wedge d\ol z_{j_q};\, 1\leq j_1<\cdots<j_q\leq n\}.\]

Let $D\subset H_{n+1}$ be an open subset. Let $\Omega^{0,q}(D)$
denote the space of smooth sections of $\Lambda^qT^{\ast0, 1}H_{n+1}$ over $D$. Let $\Omega^{0,q}_0(D)$ be the subspace of $\Omega^{0,q}(D)$ whose elements have compact support in $D$. Let $(\,\cdot\,|\,\cdot\,)$ be the $L^2$ inner product on $\Omega^{0,q}_0(H_{n+1})$ induced by $\langle\,\cdot\,|\,\cdot\,\rangle$ and the volume form $d\mu_{H_{n+1}}$ and let $\norm{\cdot}$ denote the corresponding norm. Then for all $u, v\in\Omega^{0,q}_0(H_{n+1})$
\begin{equation}
(u|v)=\int_{H_{n+1}}\langle u| v\rangle d\mu_{H_{n+1}}.
\end{equation}
Let $L^2_{(0,q)}(H_{n+1})$ be the completion of $\Omega^{0,q}_0(H_{n+1})$ with respect to $(\cdot|\cdot)$.  We write $L^2(H_{n+1}):=L^2_{(0,0)}(H_{n+1})$.
We extend $(\,\cdot\,|\,\cdot\,)$ to $L^2_{(0,q)}(H_{n+1})$
in the standard way. For $f\in L^2_{(0,q)}(H_{n+1})$, we denote $\norm{f}^2:=(\,f\,|\,f\,)$. Let
\begin{equation} \label{e-suIV}
\ddbar_b:\Omega^{0,q}(H_{n+1})\To\Omega^{0,q+1}(H_{n+1})
\end{equation}
be the tangential Cauchy-Riemann operator.
We extend
$\ddbar_{b}$ to $L^2_{(0,r)}(H_{n+1})$, $r=0,1,\ldots,n$, by
\begin{equation}\label{e-suVII}
\ddbar_{b}:{\rm Dom\,}\ddbar_{b}\subset L^2_{(0,r)}(H_{n+1})\To L^2_{(0,r+1)}(H_{n+1})\,,
\end{equation}
where ${\rm Dom\,}\ddbar_{b}:=\{u\in L^2_{(0,r)}(H_{n+1});\, \ddbar_{b}u\in L^2_{(0,r+1)}(X)\}$ and, for any $u\in L^2_{(0,r)}(H_{n+1})$, $\ddbar_{b} u$ is defined in the sense of distributions.
We also write
\begin{equation}\label{e-suVIII}
\ol{\pr}^{*}_{b}:{\rm Dom\,}\ol{\pr}^{*}_{b}\subset L^2_{(0,r+1)}(H_{n+1})\To L^2_{(0,r)}(H_{n+1})
\end{equation}
to denote the Hilbert space adjoint of $\ddbar_{b}$ in the $L^2$ space with respect to $(\,\cdot\,|\,\cdot\, )$.
Let $\Box^{(q)}_{b}$ denote the (Gaffney extension) of the Kohn Laplacian given by
\begin{equation}\label{e-suIX}
\begin{split}
{\rm Dom\,}&\Box^{(q)}_{b}\\
&=\Big\{s\in L^2_{(0,q)}(H_{n+1});\,
s\in{\rm Dom\,}\ddbar_{b}\cap{\rm Dom\,}\ol{\pr}^{*}_{b},\,
\ddbar_{b}s\in{\rm Dom\,}\ol{\pr}^{*}_{b},\ \ol{\pr}^{*}_{b}s\in{\rm Dom\,}\ddbar_{b}\Big\}\,,\\
\Box^{(q)}_{b}s&=\ddbar_{b}\ol{\pr}^{*}_{b}s+\ol{\pr}^{*}_{b}\ddbar_{b}s
\:\:\text{for $s\in {\rm Dom\,}\Box^{(q)}_{b}$}\,.
 \end{split}
\end{equation}
By a result of Gaffney, for every $q=0,1,\ldots,n$, $\Box^{(q)}_{b}$ is a positive self-adjoint operator
(see \cite[Proposition\,3.1.2]{MM}). That is, $\Box^{(q)}_{b}$ is self-adjoint and
the spectrum of $\Box^{(q)}_{b}$ is contained in $\ol\Real_+$, $q=0,1,\ldots,n$. Let
\begin{equation}\label{e-suXI-I}
S^{(q)}:L^2_{(0,q)}(H_{n+1})\To{\rm Ker\,}\Box^{(q)}_b
\end{equation}
be the orthogonal projection with respect to the $L^2$ inner product $(\,\cdot\,|\,\cdot\,)$ (Szeg\H{o} projection) and let
\begin{equation}\label{e-suXI-II}
S^{(q)}(x,y)\in D'(H_{n+1}\times H_{n+1},\Lambda^qT^{\ast 0,1}H_{n+1}\boxtimes(\Lambda^qT^{\ast 0,1}H_{n+1})^*)
\end{equation}
denote the distribution kernel of $S^{(q)}$ (Szeg\H{o} kernel). Put $\mathcal H^{q}_b(H_{n+1}):={\rm Ker\,}\Box^{(q)}_b$. Our first result is the following

\begin{theorem}\label{thm:1.1}
If $\lambda_j=0$ for some $j$, then
\[\mathcal H^{q}_b(H_{n+1})=\set{0}.\]

Suppose that  all $\lambda_j$ are non-zero and let $n_{-}$ be the number of negative $\lambda_js$ and $n_{+}$ be the number of positive $\lambda_{j}s$. If $q\notin\set{n_-,n_+}$, then
\[\mathcal H^{q}_b(H_{n+1})=\set{0}.\]
\end{theorem}

In view of Theorem~\ref{thm:1.1}, we only need to consider the non-degenerate case, that is, all $\lambda_j$ are non-zero.
We now state our explicit formula for $S^{(q)}$ in the non-degenerate case. We introduce some notations and definitions. Consider the two functions
\begin{equation}\label{phase1}
\begin{split}
&\varphi_-(x, y)\\
&=-x_{2n+1}+y_{2n+1}+i\sum_{j=1}^n|\lambda_j||z_j-w_j|^2+i\sum_{j=1}^n
\lambda_j(\overline z_jw_j-z_j\overline w_j)\in C^\infty(H_{n+1}\times H_{n+1}),\\
&\varphi_+(x, y)\\
&=x_{2n+1}-y_{2n+1}+i\sum_{j=1}^n|\lambda_j||z_j-w_j|^2+i\sum_{j=1}^n
\lambda_j(z_j\ol w_j-\ol z_jw_j)\in C^\infty(H_{n+1}\times H_{n+1}).
\end{split}
\end{equation}
Let $\varphi\in\{\varphi_-,\varphi_+\}$ be one of the functions defined above. Fix $q=0,1,\ldots,n$. For $u\in\Omega^{0,q}_0(H_{n+1})$, by using integration by parts with respect to $y_{2n+1}$ several times, we can show that
\[\lim_{\varepsilon\To0^+}\int_{H_{n+1}}\frac{1}{\bigr(-i(\varphi(x,y)+i\varepsilon)\bigr)^{n+1}}u(y)d\mu_{H_{n+1}}(y)\] exists for every $x\in H_{n+1}$,
\begin{equation}\label{e-gue170525}
\lim_{\varepsilon\To0^+}\int_{H_{n+1}}\frac{1}{\bigr(-i(\varphi(x,y)+i\varepsilon)\bigr)^{n+1}}u(y)d\mu_{H_{n+1}}(y)\in\Omega^{0,q}(H_{n+1})
\end{equation}
and the operator
\begin{equation}\label{e-gue170525I}
\begin{split}
I^{(q)}_{\varphi}: \Omega^{0,q}_0(H_{n+1})&\To\Omega^{0,q}(H_{n+1}),\\
u&\mapsto \lim_{\varepsilon\To0^+}\int_{H_{n+1}}\frac{1}{\bigr(-i(\varphi(x,y)+i\varepsilon)\bigr)^{n+1}}u(y)d\mu_{H_{n+1}}(y)\in\Omega^{0,q}(H_{n+1})
\end{split}
\end{equation}
is continuous. Moreover, we will show in Theorem~\ref{tema2} and Corollary~\ref{c-gue170528} that there is a constant $C>0$ such that
\[\norm{I^{(q)}_{\varphi}u}\leq C\norm{u},\ \ \forall u\in\Omega^{0,q}_0(H_{n+1}).\]
Thus, we can extend $I^{(q)}_{\varphi}$ to $L^2_{(0,q)}(H_{n+1})$ in the standard way and we have that
\begin{equation}\label{e-gue170525II}
I^{(q)}_{\varphi}: L^2_{(0,q)}(H_{n+1})\To L^2_{(0,q)}(H_{n+1})
\end{equation}
is continuous.
We first state our main result for positive case.

\begin{theorem}\label{main theorem 2}
Assume that $\lambda_j>0$ for every $j=1,2,\ldots,n$. With the notations used above, we have
\begin{equation}\label{e-gue170525rj}
S^{(0)}=\frac{\abs{\lambda_1}\cdots\abs{\lambda_n}}{2\pi^{n+1}}n!I^{(0)}_{\varphi_-}\ \ \mbox{on $L^2(H_{n+1})$}.
\end{equation}
\end{theorem}

We now consider non-positive case. Let $r\in\mathbb N$. For a multi-index
$J=(j_1,\ldots,j_r)\in\{1,\ldots,n\}^r$ we set $l(J)=r$. We say
that $J$ is strictly increasing if $1\leqslant
j_1<j_2<\cdots<j_r\leqslant  n$ and we put $d\ol z^J=d\ol z_{j_1}\wedge\cdots\wedge d\ol z_{j_r}$. Let  $u\in\oplus^n_{q=0}L^2_{(0,q)}(H_{n+1})$. We have the representation
\[u=\sideset{}{'}\sum_{1\leq l(J)\leq n}u_J(z)d\ol z^J+u_0(z),\ \ u_0(z)\in L^2(H_{n+1}),\]
where $\sum^{'}$ means that the summation is performed only
over strictly increasing multi-indices. Suppose that all $\lambda_j$ are non-zero and let $n_-$ be the number of negative $\lambda_js$. Assume that $n_->0$ and
suppose that $\lambda_1<0,\ldots,\lambda_{n_-}<0$.
Put
\begin{equation}\label{e-gue170525r}
\begin{split}
J_{n_-}=(1,\ldots,n_-),\\
J_{n_+}=(n_-+1,\ldots,n)\ \ \mbox{if $n_-<n$}.
\end{split}
\end{equation}
Consider the operator
\begin{equation}\label{e-gue170522a}
\begin{split}
\tau_{-}: \oplus^n_{q=0}L^2_{(0,q)}(H_{n+1})&\To L^2_{(0,n_-)}(H_{n+1}),\\
u=\sideset{}{'}\sum_{1\leq l(J)\leq n}u_Jd\ol z^J+u_0(z)&\mapsto u_{J_{n_-}}d\ol z^{J_{n_-}}.
\end{split}
\end{equation}
If $n_-<n$, set
\begin{equation}\label{e-gue170522aI}
\begin{split}
\tau_{+}: \oplus^n_{q=0}L^2_{(0,q)}(H_{n+1})&\To L^2_{(0,n-n_-)}(H_{n+1}),\\
u=\sideset{}{'}\sum_{1\leq l(J)\leq n}u_Jd\ol z^J+u_0(z)&\mapsto u_{J_{n_+}}d\ol z^{J_{n_+}}.
\end{split}
\end{equation}
If $n_-=n$, set
\begin{equation}\label{e-gue170607}
\begin{split}
\tau_{+}: \oplus^n_{q=0}L^2_{(0,q)}(H_{n+1})&\To L^2(H_{n+1}),\\
u=\sideset{}{'}\sum_{1\leq l(J)\leq n}u_Jd\ol z^J+u_0(z)&\mapsto u_0(z).
\end{split}
\end{equation}
Then \(\tau_-\) and \(\tau_{+}\) are continuous operators.
Our main result for $(0,q)$ forms is the following

\begin{theorem}\label{main theorem 3}
Suppose that all $\lambda_j$ are non-zero and let $n_-$ be the number of negative $\lambda_js$. Assume that $n_->0$.
Let $q\in\set{n_-,n_+}$, where $n_+=n-n_-$. With the notations used above, we have
\[S^{(q)}=\frac{\abs{\lambda_1}\cdots\abs{\lambda_n}}{2\pi^{n+1}}n!I^{(n_-)}_{\varphi_-}\circ\tau_-+\frac{\abs{\lambda_1}\cdots\abs{\lambda_n}}{2\pi^{n+1}}n!I^{(n_+)}_{\varphi_+}\circ\tau_+\ \ \mbox{on $L^2_{(0,q)}(H_{n+1})$}.\]
\end{theorem}

\begin{remark}\label{r-gur170525}
With the notations and assumptions used in Theorem~\ref{main theorem 3}, suppose  $q=n_-$ and $n_-\neq n_+$. Since $\tau_+u=0$ for every $u\in L^2_{(0,q)}(X)$, we get \[S^{(q)}=\frac{\abs{\lambda_1}\cdots\abs{\lambda_n}}{2\pi^{n+1}}n!I^{(n_-)}_{\varphi_-}\circ\tau_-.\]
If $q=n_-=n_+$, the operators $I^{(n_-)}_{\varphi_-}\circ\tau_-$ and $I^{(n_+)}_{\varphi_+}\circ\tau_+$ are non-trivial.
\end{remark}

\begin{remark}\label{r-gue170525I}
It should be notice that the operator $I^{(q)}_\varphi$ in \eqref{e-gue170525I} is a complex Fourier integral operator
\[\int_0^\infty e^{it\varphi(x, y)}\frac{1}{n!}t^ndt\]
with complex phase $\varphi$ and symbol $\frac{1}{n!}t^n$ (see the discussion in the beginning of Section~\ref{s-gue170528}).
\end{remark}

	In Section \ref{sec:relbergszeg} we show  how \(S^{(0)}\) is related to a weighted Bergman kernel on \(\mathbb C^n\) (see Theorem \ref{thm:szegobergman}).

\section{Preliminaries}\label{s:prelim}

We shall use the following notations: $\mathbb N=\set{1,2,\ldots}$, $\mathbb N_0=\mathbb N\cup\set{0}$, $\Real$
is the set of real numbers, $\ol\Real_+:=\set{x\in\Real;\, x\geq0}$. Let $m\in\mathbb N$.
For a multi-index $\alpha=(\alpha_1,\ldots,\alpha_m)\in\mathbb N_0^m$,
we denote by $\abs{\alpha}=\alpha_1+\ldots+\alpha_m$ its norm and by $l(\alpha)=m$ its length.
$\alpha$ is strictly increasing if $\alpha_1<\alpha_2<\ldots<\alpha_m$.

Let $z=(z_1,\ldots,z_n)$, $z_j=x_{2j-1}+ix_{2j}$, $j=1,\ldots,n$, be coordinates of $\Complex^n$. Let $\alpha=(\alpha_1,\ldots,\alpha_n)\in\mathbb N_0^n$ be a multi-index.
We write
\[
\begin{split}
&z^\alpha=z_1^{\alpha_1}\ldots z^{\alpha_n}_n\,,\quad\ol z^\alpha=\ol z_1^{\alpha_1}\ldots\ol z^{\alpha_n}_n\,,\\
&\frac{\pr}{\pr z_j}=
\frac{1}{2}\Big(\frac{\pr}{\pr x_{2j-1}}-i\frac{\pr}{\pr x_{2j}}\Big)\,,\quad
\frac{\pr}{\pr\ol z_j}=\frac{1}{2}\Big(\frac{\pr}{\pr x_{2j-1}}+i\frac{\pr}{\pr x_{2j}}\Big),\ \ j=1,\ldots,n.
\end{split}
\]
For $j, s\in\mathbb Z$, set $\delta_{j,s}=1$ if $j=s$, $\delta_{j,s}=0$ if $j\neq s$.

Let $M$ be a $C^\infty$ paracompact manifold.
We let $TM$ and $T^*M$ denote the tangent bundle of $M$
and the cotangent bundle of $M$, respectively.
The complexified tangent bundle of $M$ and the complexified cotangent bundle of $M$ will be denoted by $\Complex TM$
and $\Complex T^*M$, respectively. Write $\langle\,\cdot\,,\cdot\,\rangle$ to denote the pointwise standard pairing
between $TM$ and $T^*M$.
We extend $\langle\,\cdot\,,\cdot\,\rangle$ bilinearly to $\Complex TM\times\Complex T^*M$.
Let $G$ be a $C^\infty$ vector bundle over $M$. The fiber of $G$ at $x\in M$ will be denoted by $G_x$.
Let $E$ be a vector bundle over a $C^\infty$ paracompact manifold $M_1$. We write
$G\boxtimes E^*$ to denote the vector bundle over $M\times M_1$ with fiber over $(x, y)\in M\times M_1$
consisting of the linear maps from $E_y$ to $G_x$.  Let $Y\subset M$ be an open set.
From now on, the spaces of distribution sections of $G$ over $Y$ and
smooth sections of $G$ over $Y$ will be denoted by $D'(Y, G)$ and $C^\infty(Y, G)$, respectively.

Let $G$ and $E$ be $C^\infty$ vector
bundles over paracompact orientable $C^\infty$ manifolds $M$ and $M_1$, respectively, equipped with smooth densities of integration. If
$A: C^\infty_0(M_1,E)\To D'(M,G)$
is continuous, we write $A(x, y)$ to denote the distribution kernel of $A$.

Let $H(x,y)\in D'(M\times M_1,G\boxtimes E^*)$. We write $H$ to denote the unique
continuous operator $C^\infty_0(M_1,E)\To D'(M,G)$ with distribution kernel $H(x,y)$.
In this work, we identify $H$ with $H(x,y)$.

\section{Proof of Theorem~\ref{thm:1.1}}

Consider the affine complex space $\Complex^n$. Let $z=(z_1,\ldots,z_n)=(x_1,\ldots,x_{2n})$ be complex coordinates of $\Complex^n$, $z_j=x_{2j-1}+ix_{2j}$, $j=1,\ldots,n$. Put $d\mu(z)=2^ndx_1\cdots dx_{2n}$. The following follows from some elementary calculation. We omit the proof.

\begin{lemma}\label{lem:monomialintegrals}
Fix $q\in\{0,1,2,\ldots,n\}$. Put
\[\begin{split}
&\Td \lambda\abs{z}^2=\lambda_1|z_1|^2+\cdots+\lambda_q|z_q|^2-\lambda_{q+1}|z_{q+1}|^2-\cdots-\lambda_n|z_n|^2\ \ \mbox{if $1\leq q<n$},\\
&\Td \lambda\abs{z}^2=-\lambda_{1}|z_{1}|^2-\cdots-\lambda_n|z_n|^2\ \ \mbox{if $q=0$},\\
&\Td \lambda\abs{z}^2=\lambda_{1}|z_{1}|^2+\cdots+\lambda_n|z_n|^2\ \ \mbox{if $q=n$}.
\end{split}\]
		Consider the expression \(I(\alpha,\eta,\lambda)=\int_{\C^n}|z^\alpha|^2e^{-2\eta\Td\lambda|z|^2}d\mu(z)\) for \(\alpha\in\N^n_0\), \(\eta\in\R\).
		If $\lambda_j=0$ for some \(j\in\{1,\ldots,n\}\),  then \(I(\alpha,\eta,\lambda)=\infty\) for all \(\alpha\in\N^n_0\) and \(\eta\in\R\).
		
		Assume that all $\lambda_j$ are non-zero and let $n_{-}$ be the number of negative $\lambda_js$ and $n_{+}$ be the number of positive $\lambda_{j}s$. If
$q\not\in\{n_{-}, n_{+}\}$, then \(I(\alpha,\eta,\lambda)=\infty\) for all \(\alpha\in\N^n_0\) and \(\eta\in\R\).
	\end{lemma}
	
	We pause and introduce some notations.
Choose $\chi(\theta)\in C_0^\infty(\mathbb R)$
so that $\chi(\theta)=1$ when $|\theta|<1$ and
$\chi(\theta)=0$ when $|\theta|>2$
and set $\chi_j(\theta)=\chi(\frac{\theta}{j}),j\in\mathbb N$. For any $u(z, x_{2n+1})\in\Omega^{0, q}(H_{n+1})$ with $\|u\|<\infty$, let
\begin{equation}
\hat u_j(z, \eta)=\int_{\mathbb R} u(z, x_{2n+1})\chi_j(x_{2n+1})e^{-ix_{2n+1}\eta}dx_{2n+1}\in\Omega^{0, q}(H_{n+1}), j=1, 2,\ldots.
\end{equation}
From Parseval's formula, we have
\begin{align*}
&\int_{H_{n+1}}\!\abs{\hat u_j(z,\eta)-\hat u_t(z,\eta)}^2d\eta d\mu(z)\\
&=2\pi\int_{H_{n+1}}\!\abs{u(z,x_{2n+1})}^2\abs{\chi_j(x_{2n+1})-\chi_t(x_{2n+1})}^2dx_{2n+1} d\mu(z)\To0,\  \ \mbox{as $j,t\To\infty$}.
\end{align*}
Thus there is $\hat u(z, \eta)\in L^2_{(0, q)}(H_{n+1})$ such that $\hat u_j(z, \eta)\rightarrow \hat u(z, \eta)$ in $L^2_{(0, q)}(H_{n+1})$. We call $\hat u(z, \eta)$ the partial Fourier transform of $u(z, x_{2n+1})$ with respect to $x_{2n+1}$. Formally,
\begin{equation}
\hat u(z, \eta)=\int_{\mathbb R} e^{-i x_{2n+1} \eta}u(z, x_{2n+1})dx_{2n+1}.
\end{equation}
Moreover, we have
\begin{equation}\label{neg1}
\int_{H_{n+1}}|\hat u(z, \eta)|^2d\mu(z)d\eta=2\pi \int_{H_{n+1}}|u(z, x_{2n+1})|^2d\mu(z)dx_{2n+1}.
\end{equation}
From Fubini's theorem, $\int_{\mathbb C^n}|\hat u(z, \eta)|^2 d\mu(z)<\infty$, for a.e. $\eta\in\mathbb R$. More precisely, there is a measure zero set $A_0\subset\mathbb R$ such that $\int_{\mathbb C^n}|\hat u(z, \eta)|^2d\mu(z)<\infty, ~\forall \eta\not\in A_0.$

 Similarly, let
\begin{equation}
\check u_j(z, \eta)=\frac{1}{2\pi}\int_{\mathbb R} u(z, x_{2n+1})\chi_j(x_{2n+1})e^{ix_{2n+1}\eta}dx_{2n+1}\in\Omega^{0, q}(H_{n+1}), j=1, 2,\ldots.
\end{equation}
Then $\check u_j(z, \eta)\rightarrow \check u(z, \eta)$ in $L^2_{(0, q)}(H_{n+1})$ for some $\check u(z,\eta)$. We call $\check u(z, \eta)$ the partial inverse Fourier transform of $u(z, x_{2n+1})$ with respect to $x_{2n+1}$.

Let $u, v\in L^2_{(0, q)}(H_{n+1})$. Assume that $\int |u(z, t)|^2dt<\infty$ and $\int |u(z, t)|dt<\infty$ for all $z\in\mathbb C^n.$ Then from Parseval's formula, it is not difficult to check that
\begin{equation}\label{e-gue170527cr}
\int\int \langle \hat v(z, \eta)| u(z, \eta)\rangle d\mu(z)d\eta=\int\int \langle v(z, x_{2n+1})|\int e^{ix_{2n+1}\eta}u(z, \eta)d\eta\rangle d\mu(z)dx_{2n+1}.
\end{equation}

\begin{lemma}\label{lem:equivkernel}
Let $u=\sideset{}{'}\sum_{l(J)=q}u_Jd\overline{z}_J\in L^2_{(0,q)}(H_{n+1})$. We have

	\begin{equation}\label{CR1}
u=\sideset{}{'}\sum_{l(J)=q}u_Jd\overline{z}_J\in \mathcal{H}^q_b(H_{n+1})
	\Leftrightarrow\begin{cases}
	\left(\frac{\partial}{\partial z_j}-i\lambda_j\overline{z}_j\frac{\partial}{\partial x_{2n+1}}\right) u_J=0 &\text{, if }j\in J,\\
	\left(\frac{\partial}{\partial \overline{z}_j}+i\lambda_jz_j\frac{\partial}{\partial x_{2n+1}}\right) u_J=0 &\text{, if }j\notin J.
	\end{cases}
\end{equation}
in the sense of distribution.
\end{lemma}

\begin{proof}
Assume $u\in \mathcal H^q_b(H_{n+1})$. Then $\overline\partial_bu=0$ and $\overline\partial_b^\ast u=0.$
Choose $\chi(x)\in C_0^\infty(H_{n+1})$ such that $\chi\equiv1$ on $\{x\in H_{n+1}: |x|\leq 1\}$ and ${\rm supp}\chi\Subset\{x\in H_{n+1} : |x|<2\}.$ Set $\chi_j(x)=\chi(\frac{x}{j})$ and $v_j=u\chi_j(x)$. Then ${\rm supp}v_j\Subset \{x\in H_{n+1}: |x|<2j\}$. It is easy to see that $v_j\in{\rm Dom\,}\ddbar_b\bigcap{\rm Dom\,}\overline\partial_b^\ast $, $v_j\rightarrow u$ in $L^2_{(0, q)}(H_{n+1})$ and $\overline\partial_b v_j=\overline\partial_b\chi_j\wedge u$. Since $$\max\limits_{x\in H_{n+1}}|\overline\partial_b\chi_j(x)|=\max\limits_{x\in H_{n+1}}|\sum_k (\frac{\partial}{\partial\overline z_k}+i\lambda_kz_k\frac{\partial}{\partial x_{2n+1}})\chi(\frac{x}{j})d\overline z_k|\leq c,$$ where $c$ is a constant which does not depend on $j$. Thus, by Dominated convergence theorem $\overline\partial_b v_j\rightarrow 0$ in $L^2_{(0, q+1)}(H_{n+1})$. Similarly, $\overline\partial_b^\ast v_j\rightarrow 0$ in $L^2_{(0, q-1)}(H_{n+1}).$ By Friedrichs Lemma (see Appendix D in~\cite{CS01}), for each $j=1,2,\ldots$,  there is a $u_j\in \Omega^{0, q}_0(H_{n+1})$ such that
\begin{equation}\label{04-24}
\|u_j-v_j\|\leq \frac{1}{j},~ \|\overline\partial_b u_j-\overline\partial_b v_j\|\leq \frac{1}{j}~\text{and}~\|\overline\partial_b^\ast u_j-\overline\partial_b^\ast v_j\|\leq \frac{1}{j}.
\end{equation}
From (\ref{04-24}) we have $u_j\rightarrow u$ in $L^2_{(0,q)}(H_{n+1})$, $\overline\partial_b u_j\rightarrow 0$ in $L^2_{(0,q+1)}(H_{n+1})$ and $\overline\partial_b^\ast u_j\rightarrow 0$ in $L^2_{(0,q-1))}(H_{n+1})$. We deduce that
\begin{equation}\label{e-gue170527rcm}
(\,\Box^q_b u_j\,|\,u_j\,)=(\,\overline\partial_b u_j\,|\,\overline\partial_bu_j\,)+(\,\overline\partial_b^\ast v_j\,|\,\overline\partial_b^\ast v_j\,)\rightarrow0\ \ \mbox{as $j\To\infty$}.
\end{equation}
Write $u_j=\sideset{}{'}\sum_{l(J)=q}u_{j,J}d\ol z^J$.
It is well-known that (see Chapter 10 in~\cite{CS01})
\begin{equation}\label{e-gue170527rm}
\Box^q_b u_j=-\sideset{}{'}\sum_{l(J)=q}((\sum_{k\not\in J}Z_k\overline Z_k+\sum_{k\in J}\overline Z_k Z_k)u_{jJ})d\overline z_J,
\end{equation}
where $Z_k=\frac{\partial}{\partial z_k}-i\lambda_j\overline{z}_k\frac{\partial}{\partial x_{2n+1}}$, $k=1,\ldots,n$. From \eqref{e-gue170527rcm} and \eqref{e-gue170527rm},
we have
\begin{equation}\label{04-24-2}
\sum_{k\not\in J}\|\overline Z_k u_{jJ}\|^2+\sum_{k\in J}\|Z_k u_{jJ}\|^2\rightarrow 0,\ \ \mbox{as $j\rightarrow\infty$},\ \ \forall J, l(J)=q.
\end{equation}
From (\ref{04-24-2}), the direction ``$\Rightarrow$'' in (\ref{CR1}) follows.

Let $u=\sideset{}{'}\sum_{l(J)=q}u_Jd\overline{z}_J\in L^2_{(0,q)}(H_{n+1})$ be arbitrary. Assume that
\begin{equation}\label{e-gue170527ry}
\left(\frac{\partial}{\partial \overline{z}_j}+i\lambda_jz_j\frac{\partial}{\partial x_{2n+1}}\right) u_J=0\ \ \text{, if }j\notin J
\end{equation}
and
\begin{equation}\label{e-gue170527ryI}
\left(\frac{\partial}{\partial z_j}-i\lambda_j\ol z_j\frac{\partial}{\partial x_{2n+1}}\right) u_J=0\ \ \text{if }j\in J.
\end{equation}
We have
\begin{equation}\label{e-gue170527ryII}
\ddbar_bu=\sideset{}{'}\sum_{l(J)=q}\sum_{1\leq j\leq n, j\notin J}\left(\frac{\partial}{\partial \overline{z}_j}+i\lambda_jz_j\frac{\partial}{\partial x_{2n+1}}\right)u_Jd\ol z_j\wedge d\overline{z}_J.
\end{equation}
From \eqref{e-gue170527ryII} and \eqref{e-gue170527ry}, we get $\ddbar_bu=0$. Moreover, from \eqref{e-gue170527ryI}, it is easy to see that
\begin{equation}\label{e-gue170527ryIII}
(\,u\,|\,\ddbar_bv\,)=0,\ \ \forall v\in\Omega^{0,q-1}_0(H_{n+1}).
\end{equation}
Let $v\in{\rm Dom\,}\ddbar_b\bigcap L^2_{(0,q-1)}(H_{n+1})$. By using Friedrichs Lemma, we can find $v_j\in\Omega^{0,q-1}_0(H_{n+1})$, $j=1,2,\ldots$, such that
$v_j\To v\in L^2_{(0,q-1)}(H_{n+1})$ and $\ddbar_bv_j\To\ddbar_bv$ in $L^2_{(0,q)}(H_{n+1})$. From this observation and \eqref{e-gue170527ryIII}, we see that
\[(\,u\,|\,\ddbar_bv\,)=0\ \ \forall v\in{\rm Dom\,}\ddbar_b\bigcap L^2_{(0,q-1)}(H_{n+1})\]
and hence
$\ol{\pr}^*_bu=0$. The direction ``$\Leftarrow$'' follows. We get the conclusion of Lemma \ref{lem:equivkernel}.
\end{proof}

\begin{lemma}\label{fo1}
Given $u=\sideset{}{'}\sum_{l(J)=q}u_Jd\overline{z}_J\in\mathcal{H}^q_b(H_{n+1})$  we have
\begin{equation}\label{f1}
\begin{cases}
	\left(\frac{\partial}{\partial z_j}+\lambda_j\overline{z}_j\eta\right) \hat{u}_J(z,\eta)=0 &\text{if }j \in J,\\
	\left(\frac{\partial}{\partial \overline{z}_j}-\lambda_jz_j\eta\right) \hat{u}_J(z,\eta)=0 &\text{if }j\notin J
	\end{cases}
\end{equation}
for a.e. $\eta\in\Real$ and all $J$, $l(J)=q$,  in the sense of distribution.
\end{lemma}

\begin{proof}
Let $A_0\subset\mathbb R$ be as in the discussion after  (\ref{neg1}). Let $\varphi\in C^\infty_0(\mathbb C^n)$. Fix $l(J)=q$ and fix $j=1,\ldots,n$ with $j\not\in J$. Put $h(\eta)=-\int_{\mathbb C^n}\hat u_J(z, \eta)\overline{(\frac{\partial}{\partial z_j}+\lambda_j\overline z_j\eta) \varphi(z)}d\mu(z)$ if $\eta\not\in A_0$, $h(\eta)=0$ if $\eta\in A_0$. We can check that
 \begin{equation}\label{4-24-3}
 |h(\eta)|^2\leq\int_{\mathbb C^n}|\hat u_J(z, \eta)|^2d\mu(z)\int_{\mathbb C^n}|(\frac{\partial}{\partial z_j}+\lambda_j\overline z_j\eta) \varphi(z)|^2d\mu(z).
 \end{equation}
 For $R>0$, put $g_{R}(\eta)=h(\eta)\chi_{[-R, R]}(\eta)$, where $\chi_{[-R, R]}(\eta)=1$ if $-R\leq \eta\leq R$ and $\chi_{[-R, R]}(\eta)=0$ if $|\eta|>R.$ From (\ref{4-24-3}), we have
 \begin{equation}
 \int |g_R(\eta)|^2d\eta=\int_{-R}^R|h(\eta)|^2d\eta\leq C_R\|\hat u_J(z, \eta)\|^2<\infty,
 \end{equation}
 where $C_R>0$ is a constant.
 Thus, $g_R(\eta)\in L^2(\mathbb R)\cap L^1(\mathbb R)$. From \eqref{e-gue170527cr}, we have
 \begin{equation}
 \begin{split}
 &\int_{-R}^R|h(\eta)|^2d\eta=\int_{\mathbb R} h(\eta)\overline g_{R}(\eta)d\eta\\
 &=\int\int\hat u_J(z, \eta)\overline{(\frac{\partial}{\partial z_j}+\lambda_j\overline z_j\eta) \varphi(z)g_R(\eta)} d\eta d\mu(z)\\
 &=\int\int u_J(z, x_{2n+1})\overline{\int_{\mathbb R}e^{ix_{2n+1}\eta}(\frac{\partial}{\partial z_j}+\lambda_j\overline z_j\eta) \varphi(z)g_R(\eta) d\eta}d\mu(z)dx_{2n+1}\\
 &=\int_{H_{n+1}} u_J(z, x_{2n+1})\overline{\left(\frac{\partial}{\partial z_j}-i\lambda_j\overline z_j\frac{\partial}{\partial x_{2n+1}}\right)\left(\varphi(z)\int_{\mathbb R}e^{ix_{2n+1}\eta}g_R(\eta)d\eta\right)}d\mu_{H_{n+1}}.
 \end{split}
 \end{equation}
 Put $S(z, x_{2n+1})=\varphi(z)\int_{\mathbb R}e^{ix_{2n+1}\eta}g_R(\eta)d\eta$. Let $u_j=\sideset{}{'}\sum_{l(J)=q}u_{jJ}d\ol z_J$, $j=1,2,\ldots$, be the sequence chosen in (\ref{04-24}). Then
 \begin{equation}\label{e-gue170527yc}
 \begin{split}
\int_{-R}^R|h(\eta)|^2d\eta
&=\lim_{k\rightarrow\infty}\int_{H_{n+1}}u_{kJ}\overline Z_j \overline {S(z, x_{2n+1})}d\mu_{H_{n+1}}\\
&=-\lim_{k\rightarrow\infty}\int_{H_{n+1}}\overline Z_j u_{k J}\overline{S(z, x_{2n+1})}d\mu_{H_{n+1}}.
\end{split}
 \end{equation}
 From (\ref{04-24-2}), \eqref{e-gue170527yc} and H\"older's inequality, we conclude that $\int_{-R}^R|h(\eta)|^2d\eta=0.$ Letting $R\rightarrow\infty$, we get $h(\eta)=0$ almost everywhere. Thus, we have proved that $\forall j\not\in J$ and for a given $\varphi(z)\in C_0^\infty(\mathbb C^n)$, $\int_{\mathbb C^n}\hat u_J(z, \eta)\overline{(\frac{\partial}{\partial z_j}+\lambda_j\overline z_j\eta) \varphi(z)}d\mu(z)=0$ for a.e. $\eta\in\Real$.

 Let us consider the Sobolev space $W^1(\mathbb C^n)$ of distributions in $\mathbb C^n$ whose derivatives of order $\leq 1$ are in $L^2.$ Since $W^1(\mathbb C^n)$ is separable and $C_0^\infty(\mathbb C^n)$ is dense in $W^1(\mathbb C^n)$, we can find $f_k\in C_0^\infty(\mathbb C^n)$, $k=1,2,\ldots$, such that $\{f_k\}^\infty_{k=1}$ is a dense subset of $W^1(\mathbb C^n)$. Moreover, we can find $\{f_k\}^\infty_{k=1}$ so that for all $g\in C_0^\infty(\mathbb C^n)$ with ${\rm supp}g\Subset B_r:=\{z\in\mathbb C^n: |z|<r\}, r>0$, we can find $f_{k_1}, f_{k_2}, \ldots, {\rm supp}f_{k_t}\Subset B_r, t=1, 2, \ldots, $ such that $f_{k_t}\rightarrow g$ in $W^1(\mathbb C^n)$, as $t\To\infty$.

 Now, for each $k$, we can repeat the method above and find a measurable set $A_k\supseteqq A_0$, $|A_k|=0$ such that
 $\int_{\mathbb C^n}\hat u_J(z, \eta)\overline{(\frac{\partial}{\partial z_j}+\lambda_j\overline z_j\eta) f_k(z)}d\mu(z)=0, ~\forall \eta\not\in A_k.$
 Put $A=\cup_k A_\lambda.$ Then $|A|=0$ and for all $\eta\not\in A$ and all $k$
 $$\int_{\mathbb C^n}\hat u_J(z, \eta)\overline{(\frac{\partial}{\partial z_j}+\lambda_j\overline z_j\eta) f_k(z)}d\mu(z)=0.$$
 Let $\varphi\in C_0^\infty(\mathbb C^n)$ with ${\rm supp}\varphi\Subset B_r.$ From the discussion above, we can find $f_{k_1}, f_{k_2}, \ldots, {\rm supp} f_{k_t}\Subset B_r$, $t=1, 2, \ldots, $ such that $f_{k_t}\rightarrow \varphi$ in $W^1(\mathbb C^n)$ , as $t\rightarrow \infty$. Then for $\eta\not\in A$,
 \begin{equation}
 \int_{\mathbb C^n}\hat u_J(z, \eta)\overline{(\frac{\partial}{\partial z_j}+\lambda_j\overline z_j\eta) \varphi(z)}d\mu(z)=\int_{\mathbb C^n}\hat u_J(z, \eta)\overline{(\frac{\partial}{\partial z_j}+\lambda_j\overline z_j\eta) (\varphi(z)-f_{k_t}(z))}d\mu(z).
 \end{equation}
 Fix $\eta\notin A$. By H\"older's inequality
 \begin{equation}
 \begin{split}
 &\left|\int_{\mathbb C^n}\hat u_J(z, \eta)\overline{(\frac{\partial}{\partial z_j}+\lambda_j\overline z_j\eta) (\varphi(z)-f_{k_t}(z))}d\mu(z)\right|\\
 &\leq C(\eta)\sum_{|\alpha|\leq1}\int_{\mathbb C^n}|\partial_{x}^\alpha(\varphi-f_{k_t})|^2d\mu(z)\rightarrow0,\ \ \mbox{as $t\rightarrow\infty$},
 \end{split}
 \end{equation}
 where $C(\eta)>0$ is a constant.
 Thus, for all $\eta\not\in A$,
 \begin{equation}
 \int_{\mathbb C^n}\hat u_J(z, \eta)\overline{(\frac{\partial}{\partial z_j}+\lambda_j\overline z_j\eta) \varphi(z)}d\mu(z)=0, ~\forall \varphi\in C_0^\infty(\mathbb C^n).
 \end{equation}
 We have proved the second case of (\ref{f1}) and the proof of the first case of (\ref{f1}) is the same. The lemma follows.
\end{proof}

It should be mentioned that the partial Fourier transform technique used in the proof of Lemma~\ref{fo1} was inspired by~\cite{HM12}.

\begin{proof}[Proof of Theorem~\ref{thm:1.1}]
Fix $q=0,1,\ldots,n$. Assume that
\begin{equation}\label{e-gue170528ry}
\begin{split}
&\mbox{$\lambda_j=0$ for some $j\in\set{1,\ldots,n}$}\\
&\mbox{ or all $\lambda_j$ are non-zero but $q\notin\set{n_-,n_+}$},\
\end{split}
\end{equation}
where $n_{-}$ denotes the number of negative $\lambda_js$ and $n_{+}$ denotes the number of positive $\lambda_{j}s$.
	Consider \(u=\sum_{l(J)=q}'u_Jd\overline{z}_J\in \mathcal{H}^q_b(H_{n+1})\). For every \(l(J)=q\), we can find
	\[(((z,x_{2n+1})\mapsto u_J(z,x_{2n+1}))\in L^2(H_{n+1})\] and hence
	\[((z,\eta)\mapsto \hat{u}_J(z,\eta))\in L^2(H_{n+1}),\]
	where \(\hat{u}_J\) is the partial Fourier transform of \(u_J\) with respect to \(x_{2n+1}\). Since \(u\in \mathcal{H}^q_b(H_{n+1})\), by Lemma \ref{fo1}, for a.e. $\eta\in\Real$, we have
	\begin{align}\label{eq:fouriereq}\begin{cases}
	\left(\frac{\partial}{\partial z_j}+\lambda_j\overline{z}_j\eta\right) \hat{u}_J=0 &\text{if }j \in J, \\
	\left(\frac{\partial}{\partial \overline{z}_j}-\lambda_jz_j\eta\right) \hat{u}_J=0 &\text{if }j\notin J,
	\end{cases}
	\end{align}
	for all \(l(J)=q\), $j=1,\ldots,n$.  We will show \(u_J=0\) for all \(l(J)=q\).
	So let \(J_0\), \(|J_0|=q\), be arbitrary. Without loss of generality we can assume \(J_0=\{1,\ldots,q\}\). In order to simplify the notation we write
	\[\begin{split}
&\Td \lambda\abs{z}^2=\lambda_1|z_1|^2+\ldots+\lambda_q|z_q|^2-\lambda_{q+1}|z_{q+1}|^2-\ldots-\lambda_n|z_n|^2\ \ \mbox{if $q\geq1$},\\
&\Td \lambda\abs{z}^2=-\lambda_{1}|z_{1}|^2-\ldots-\lambda_n|z_n|^2\ \ \mbox{if $q=0$}.
\end{split}\]
Then (\ref{eq:fouriereq}) reduces to
	\begin{equation}\label{fourier transform}
\begin{cases}
	\frac{\partial}{\partial z_j}\left(e^{\eta\Td\lambda|z|^2}\hat{u}_{J_0}(z,\eta)\right) =0 &\mbox{if $j\in J_0$,  for a.e. $\eta\in\mathbb R$},\\
	\frac{\partial}{\partial \overline{z}_j}\left(e^{\eta\Td\lambda|z|^2}\hat{u}_{J_0}(z,\eta)\right) =0 &\mbox{if $j\notin J_0$, for a.e. $\eta\in\mathbb R$}.
	\end{cases}\end{equation}
	For every $\eta\in\Real$, let
	\[\begin{split}
	F_{J_0}(z_1,\ldots,z_{n},\eta):=e^{\eta\Td\lambda|z|^2}\hat{u}_{J_0}(\ol z_1,\ldots,\ol z_q,z_{q+1},\ldots,z_n,\eta)\ \ \mbox{if $q\geq1$},\\
	F_{J_0}(z_1,\ldots,z_{n},\eta):=e^{\eta\Td\lambda|z|^2}\hat{u}_{J_0}(z_1,\ldots,z_q,z_{q+1},\ldots,z_n,\eta)\ \ \mbox{if $q=0$}.
	\end{split}\]
	From \eqref{fourier transform}, it is easy to see that \((z,\eta)\mapsto F_{J_0}(z,\eta)\) is holomorphic in \(z\) for almost every \(\eta\) and we have
	\begin{equation}\label{e-gue170608c}
	\int_{H_{n+1}}|F_{J_0}(z,\eta)|^2e^{-2\eta\Td\lambda|z|^2}d\mu(z)d\eta=\int_{H_{n+1}}|\hat{u}_{J_0}(z,\eta)|^2d\mu(z)d\eta<\infty
	\end{equation}
	 which implies
	 \begin{equation}\label{e-gue170528}
	 \mbox{$\int_{\C^n}|F_{J_0}(z,\eta)|^2e^{-2\eta\Td\lambda|z|^2}d\mu(z)<\infty$ for almost every $\eta$}.
	 \end{equation}
	 Let $B$ be a negligible set of $\Real$ such that for all $\eta\notin B$, $F_{J_0}(z,\eta)$ is holomorphic in $z$ and \eqref{e-gue170528} holds. For $\eta\notin B$, we have
	 \[F_{J_0}(z,\eta)=\sum_{\alpha\in\mathbb N^n_0}F_{J_0,\alpha}(\eta)z^\alpha.\]
	 Fix $\eta\notin B$ and fix a $\alpha_0\in\mathbb N^n_0$. It is easy to see that
	\[|F_{J_0,\alpha_0}(\eta)|^2I(\alpha_0,\eta,\lambda)\leq \int_{\C^n}|F_{J_0}(z,\eta)|^2e^{-2\eta\Td\lambda|z|^2}d\mu(z)<\infty,\]
	where \(I(\alpha_0,\eta,\lambda)=\int_{\C^n}|z^{\alpha_0}|^2e^{-2\eta\Td\lambda|z|^2}d\mu(z)\). Under the assumption~\ref{e-gue170528ry} and using Lemma \ref{lem:monomialintegrals}, we find \(I(\alpha_0,\eta,\lambda)=\infty\) for all \(\eta\in\R\) which implies \(F_{J_0,\alpha_0}(\eta)=0\). Thus,   for all \(\alpha\in\N^n_0\),  \(F_{J_0,\alpha}(\eta)=0\) and hence \(\int_{\C^n}|F_{J_0}(z,\eta)|^2e^{-2\eta\Td\lambda|z|^2}d\mu(z)=0\) for almost every \(\eta\in\R\). By Fubini's theorem, \eqref{neg1} and \eqref{e-gue170608c}, we get
	\[\|u_{J_0}\|^2=\frac{1}{2\pi}\|\hat{u}_{J_0}\|^2=\int_{H_{n+1}}|F_{J_0}(z,\eta)|^2e^{-2\eta\Td\lambda|z|^2}d\mu(z)d\eta=0.\]
We have proved that  \(u_J=0\), for all $l(J)=q$. Thus, $\mathcal{H}^q_b(H_{n+1})=\set{0}$ and Theorem~\ref{thm:1.1} follows.
	\end{proof}

\section{Complex Fourier integral operators}\label{s-gue170528}

Let $\varphi\in\{\varphi_-,\varphi_+\}$ be one of the two functions defined in \eqref{phase1}, that is
 \begin{equation*}
 \begin{split}
 &\varphi_-(x, y)\\
 &=-x_{2n+1}+y_{2n+1}+i\sum_{j=1}^n|\lambda_j||z_j-w_j|^2+i\sum_{j=1}^n
 \lambda_j(\overline z_jw_j-z_j\overline w_j)\in C^\infty(H_{n+1}\times H_{n+1}),\\
 &\varphi_+(x, y)\\
 &=x_{2n+1}-y_{2n+1}+i\sum_{j=1}^n|\lambda_j||z_j-w_j|^2+i\sum_{j=1}^n
 \lambda_j(z_j\ol w_j-\ol z_jw_j)\in C^\infty(H_{n+1}\times H_{n+1}).
 \end{split}
 \end{equation*}
Take $\chi\in C^\infty_0(\Real)$ with $\chi(x)=1$ if $\abs{x}\leq 1$ and $\chi(x)=0$ if $\abs{x}>2$. Fix $q=0,1,\ldots,n$. For $u\in\Omega^{0,q}_0(H_{n+1})$, by using integration by parts with respect to $y_{2n+1}$ several times, we can show that
\[\lim_{\varepsilon\To0^+}\int^\infty_0\int_{H_{n+1}}e^{it\varphi(x,y)}t^n\chi(\varepsilon t)u(y)d\mu_{H_{n+1}}(y)dt\] exists for every $x\in H_{n+1}$,
\begin{equation}\label{e-gue170528ay}
\lim_{\varepsilon\To0^+}\int^\infty_0\int_{H_{n+1}}e^{it\varphi(x,y)}t^n\chi(\varepsilon t)u(y)d\mu_{H_{n+1}}(y)dt\in\Omega^{0,q}(H_{n+1})
\end{equation}
and the operator
\begin{equation}\label{e-gue170528ayI}
\begin{split}
\int^\infty_0e^{it\varphi(x,y)}t^ndt: \Omega^{0,q}_0(H_{n+1})&\To\Omega^{0,q}(H_{n+1}),\\
u&\mapsto \lim_{\varepsilon\To0^+}\int^\infty_0\int_{H_{n+1}}e^{it\varphi(x,y)}t^n\chi(\varepsilon t)u(y)d\mu_{H_{n+1}}(y)dt
\end{split}
\end{equation}
is continuous. The operator $\int^\infty_0e^{it\varphi(x,y)}t^ndt$ is a complex Fourier integral operator in the sense of \cite{GrSj94}. Again, by using integration by parts with respect to $y_{2n+1}$ several times, we can show that
\begin{equation}\label{e-gue170528ayII}
\begin{split}
&\lim_{\varepsilon\To0^+}\int_{H_{n+1}}\int_0^\infty e^{-t\bigr(-i(\varphi+i\varepsilon)\bigr)}t^nu(y)dtd\mu_{H_{n+1}}(y)\\
&=\lim_{\varepsilon\To0^+}\int^\infty_0\int_{H_{n+1}}e^{it\varphi(x,y)}t^n\chi(\varepsilon t)u(y)d\mu_{H_{n+1}}(y)dt.
\end{split}
\end{equation}
Note that
\begin{equation}\label{e-gue170528ayIII}
\int_0^\infty e^{-tx}t^mdt=m! x^{-m-1}\ \ \text{if}~ m\in\mathbb Z, m\geq0.
\end{equation}
From \eqref{e-gue170528ayIII} and \eqref{e-gue170528ayII}, we deduce that
\[\begin{split}
&\lim_{\varepsilon\To0^+}\int_{H_{n+1}}\frac{1}{\bigr(-i(\varphi(x,y)+i\varepsilon)\bigr)^{n+1}}u(y)d\mu_{H_{n+1}}(y)\\
&=\lim_{\varepsilon\To0^+}\frac{1}{n!}\int^\infty_0\int_{H_{n+1}}e^{it\varphi(x,y)}t^n\chi(\varepsilon t)u(y)d\mu_{H_{n+1}}(y)dt\end{split}\]
and hence $I^{(q)}_\varphi=\frac{1}{n!}\int^\infty_0e^{it\varphi(x,y)}t^ndt$ on $\Omega^{0,q}_0(H_{n+1})$, where $I^{(q)}_\varphi$ is given by \eqref{e-gue170525I}.
We need the following

\begin{theorem}\label{tema2}
There is a constant $C>0$ such that for all $\varepsilon>0$ and $u\in\Omega^{0,q}_0(H_{n+1})$, we have
\[\int_{H_{n+1}}\abs{\int_{H_{n+1}}\int^\infty_0\chi(\varepsilon t)e^{it\varphi(x,y)}t^nu(y)d\mu_{H_{n+1}}(y)dt}^2d\mu_{H_{n+1}}(x)\leq C\norm{u}^2.\]
\end{theorem}

\begin{proof}
We may assume that $\varphi=\varphi_-$. For the case $\varphi=\varphi_+$, the proof is the same.
Let $u\in\Omega^{0,q}_0(H_{n+1})$ be a compactly supported \((0,q)\)-form and set
\[\check u(y',t):=\frac{1}{2\pi}\int u(y',y_{2n+1})e^{ity_{2n+1}}dy_{2n+1},\]
 where $y'=(y_1,\ldots,y_{2n})$
By Parseval's formula, we have
\begin{equation}\label{e-I}
\int\abs{\check u(y',t)}^2dt=\frac{1}{2\pi}\int\abs{u(y',y_{2n+1})}^2dy_{2n+1}.
\end{equation}
Let
\[g_\varepsilon(z,t):=\int \chi(\varepsilon t)\check u(y',t)\chi_{[0,\infty)}(t)e^{-t\abs{\lambda}\abs{z-w}^2-t\lambda(\ol zw-z\ol w)}d\mu(y'), \]
where $\chi_{[0,\infty)}(t)=1$ if $t\in[0,\infty)$, $\chi_{[0,\infty)}(t)=0$ if $t\notin[0,\infty)$, $\abs{\lambda}\abs{z-w}^2:=\sum^n_{j=1}\abs{\lambda_j}\abs{z_j-w_j}^2$, $i\lambda(\ol zw-z\ol w)=i\sum^n_{j=1}\lambda_j(\ol z_jw_j-z_j\ol w_j)$ and $d\mu(y')=2^ndy_1\cdots dy_{2n}$. Then we find
\begin{equation}\label{e-II}
\int_{H_{n+1}}\int^\infty_0\chi(\varepsilon t)e^{it\varphi(x,y)}t^nu(y)d\mu_{H_{n+1}}(y)dt=(2\pi)\int t^ng_\varepsilon(z,t)e^{-itx_{2n+1}}dt.
\end{equation}
By Parseval's formula again, we have
\begin{equation}\label{e-III}
\int\abs{\int t^ng_\varepsilon(z,t)e^{-itx_{2n+1}}dt}^2dx_{2n+1}=(2\pi)\int t^{2n}\abs{g_\varepsilon(z,t)}^2dt.
\end{equation}
From \eqref{e-I}, \eqref{e-II} and \eqref{e-III}, we have
\begin{equation}
\begin{split}
&\int_{H_{n+1}}\abs{\int^\infty_0\chi(\varepsilon t)e^{it\varphi(x,y)}t^nu(y)d\mu_{H_{n+1}}(y)dt}^2d\mu(z)dx_{2n+1}\\
&=(4\pi)^2\int_{H_{n+1}}\abs{\int t^ng_\varepsilon(z,t)e^{-itx_{2n+1}}dt}^2dx_{2n+1}d\mu(z)\\
&=(8\pi)^3\int_{H_{n+1}} t^{2n}\abs{g_\varepsilon(z,t)}^2d\mu(z)dt\\
&=(8\pi)^3\int_{H_{n+1}} t^{2n}\chi^2(\varepsilon t)\abs{\int_{\C^n}\check u(y',t)\chi_{[0,\infty)}(t)e^{-t\abs{\lambda}\abs{z-w}^2-t\lambda(\ol zw-z\ol w)}d\mu(y')}^2d\mu(z)dt\\
&\leq C_1\int_{H_{n+1}} t^{2n}\chi^2(\varepsilon t)\Bigr(\int_{\C^n}\abs{\check u(y',t)\chi_{[0,\infty)}(t)}^2e^{-t\abs{\lambda}\abs{z-w}^2}d\mu(w)\int_{\C^n} e^{-t\abs{\lambda}\abs{z-w}^2}d\mu(w)\Bigr)d\mu(z)dt\\
&\leq C_2 \int_{H_{n+1}}\int_{\C^n} t^{n} \abs{\chi(\varepsilon t)\check u(y',t)\chi_{[0,\infty)}(t)}^2e^{-t\abs{\lambda}\abs{z-w}^2}d\mu(z)d\mu(w)dt\\
&\leq C_3 \int_{H_{n+1}}\abs{\check u(y',t)\chi_{[0,\infty)}(t)}^2d\mu(w)dt\\
&\leq C_3\int_{H_{n+1}}\abs{\check u(y',t)}^2d\mu(w)dt=\frac{C_3}{2\pi}\int_{H_{n+1}}\abs{u(y',y_{2n+1})}^2d\mu(y),
\end{split}
\end{equation}
where $C_1>0, C_2>0, C_3>0$ are constants independent of $\varepsilon$ and $u$.
The theorem follows.
\end{proof}

From Theorem~\ref{tema2}, we deduce

\begin{corollary}\label{c-gue170528}
There is a constant $C>0$ such that
\[\norm{I^{(q)}_{\varphi}u}\leq C\norm{u},\ \ \forall u\in\Omega^{0,q}_0(H_{n+1}).\]
Thus, we can extend $I^{(q)}_{\varphi}$ to $L^2_{(0,q)}(H_{n+1})$ in the standard way and we have that
\[I^{(q)}_{\varphi}: L^2_{(0,q)}(H_{n+1})\To L^2_{(0,q)}(H_{n+1})\]
is continuous.
\end{corollary}

\section{Proof of  Theorem~\ref{main theorem 2}}\label{s-gue170528e}

In this Section, we will prove Theorem~\ref{main theorem 2}.  We assume that $\lambda_j>0$, for all $j=1,2,\ldots,n$. Put
\begin{equation}\label{e-gue170528xr}
\tilde S^{(0)}:=\frac{\abs{\lambda_1}\cdots\abs{\lambda_n}}{2\pi^{n+1}}n!I^{(0)}_{\varphi_-}.
\end{equation}

\begin{lemma}\label{l-gue170528xr}
For every $u\in L^2(H_{n+1})$, we have $\tilde S^{(0)}u\in\mathcal H^{0}_b(H_{n+1})$.
\end{lemma}

\begin{proof}
Let $\chi\in C^\infty_0(\Real)$ be as in the beginning of Section~\ref{s-gue170528}. Fix $\varepsilon>0$. We first show that  for all $u\in C^\infty_0(H_{n+1})$, we have
\begin{equation}\label{CR1a}
I^{(0)}_{\varphi_-,\varepsilon}u:=\frac{1}{n!}\int_{H_{n+1}}\int_0^\infty e^{it\varphi_-(x, y)}t^nu(y)\chi(\varepsilon t)dtd\mu_{H_{n+1}}(y)\in \mathcal H^{0}_b(H_{n+1}).
\end{equation}
Let  $u\in C^\infty_0(H_{n+1})$. For every $j=1, \cdots, n$, we have
\begin{equation}\label{CCR2}
\begin{split}
&\overline Z_j\left(\int_{H_{n+1}}\int_0^\infty e^{it\varphi_-(x, y)}t^nu(y)\chi(\varepsilon t)dt d\mu_{H_{n+1}}(y)\right)\\
=&\int_{H_{n+1}}\left(\frac{\partial}{\partial\overline z_j}+i\lambda_j z_j\frac{\partial}{\partial x_{2n+1}}\right)\varphi_-(x, y)\int_0^\infty it e^{it\varphi_-(x, y)}t^nu(y)\chi(\varepsilon t)dt d\mu_{H_{n+1}}(y)\\
=&\int_{H_{n+1}}\left[i\abs{\lambda_j}(z_j-w_j)+i\lambda_jw_j-i\lambda_jz_j\right]
\int_0^\infty it e^{it\varphi_-(x, y)}t^nu(y)\chi(\varepsilon t)dt d\mu_{H_{n+1}}(y)\\
=&0.
\end{split}
\end{equation}
Thus, we get the conclusion of (\ref{CR1a}). Since $\lim_{\varepsilon\To0^+}I^{(0)}_{\varphi_-,\varepsilon}u=I^{(0)}_{\varphi_-}u$ in $C^\infty(H_{n+1})$ topology, we deduce that
$I^{(0)}_{\varphi_-}u\in\mathcal H^{0}_b(H_{n+1})$, for every $u\in C^\infty_0(H_{n+1})$.

Let $u\in L^2(H_{n+1})$ and take $u_j\in C^\infty_0(H_{n+1})$, $j=1,2,\ldots$, $u_j\To u$ in $L^2(H_{n+1})$ as $j\To\infty$.  From Theorem~\ref{tema2}, we see that
$I^{(0)}_{\varphi_-}u_j\To I^{(0)}_{\varphi_-}u$ $L^2(H_{n+1})$ as $j\To\infty$ and hence $\ddbar_b(I^{(0)}_{\varphi_-}u_j)\To\ddbar_b(I^{(0)}_{\varphi_-}u)$ in $D'(H_{n+1})$ as $j\To\infty$. Thus, $I^{(0)}_{\varphi_-}u\in{\rm Ker\,}\ddbar_b=\mathcal H^{0}_b(H_{n+1})$. The lemma follows.
\end{proof}

We need

\begin{lemma}\label{l-gue170602}
Let $g\in C^\infty_0(H_{n+1})$ and $u\in L^2(H_{n+1})$. Then,
\[\begin{split}
&(\Td S^{(0)}u\,|\,g\,)\\
&=\frac{\abs{\lambda_1}\cdots\abs{\lambda_n}}{2\pi^{n+1}}\int^\infty_0t^n\hat u(w,-t)\ol{\hat g(z,-t)}e^{-t\abs{\lambda}\abs{z-w}^2-t\lambda(\ol zw-z\ol w)}d\mu(z)d\mu(w)dt,\end{split}\]
where $\abs{\lambda}\abs{z-w}^2=\sum^n_{j=1}\abs{\lambda_j}\abs{z-w}^2$, $\lambda(\ol zw-z\ol w)=\sum^n_{j=1}\lambda_j(\ol z_jw_j-z_j\ol w_j)$.
\end{lemma}

\begin{proof}
Let $u_j\in C^\infty_0(H_{n+1})$, $j=1,2,\ldots$, with $\lim_{j\To\infty}u_j\To u$ in $L^2(H_{n+1})$ as $j\To\infty$. We have
\begin{equation}\label{e-gue170602}
\lim_{j\To\infty}(\Td S^{(0)}u_j\,|\,g\,)=(\Td S^{(0)}u\,|\,g\,).
\end{equation}
Let $\chi\in C^\infty_0(\Real)$ be as in the beginning of Section~\ref{s-gue170528}. We have
\begin{equation}\label{e-gue170602I}
\begin{split}
(\Td S^{(0)}u_j\,|\,g\,)&=\lim_{\varepsilon\To0^+}c_0\int e^{it(-x_{2n+1}+y_{2n+1}+\hat\varphi(z,w))}t^n\chi(\varepsilon t)u_j(y)\ol g(x)d\mu_{H_{n+1}}(y)d\mu_{H_{n+1}}(x)dt\\
&=\lim_{\varepsilon\To0^+}c_0\int^\infty_0t^n\chi(\varepsilon t)\hat u_j(w,-t)\ol{\hat g(z,-t)}e^{it\hat\varphi(z,w)}d\mu(z)d\mu(w)dt\\
&=c_0\int^\infty_0t^n\hat u_j(w,-t)\ol{\hat g(z,-t)}e^{it\hat\varphi(z,w)}d\mu(z)d\mu(w)dt,
\end{split}
\end{equation}
where $c_0=\frac{\abs{\lambda_1}\cdots\abs{\lambda_n}}{2\pi^{n+1}}$ and $\hat\varphi(z,w)=i\abs{\lambda}\abs{z-w}^2+i\lambda(\ol zw-z\ol w)$. We have
\begin{equation}\label{e-gue170602II}
\begin{split}
&\abs{\int^\infty_0t^n\Bigr(\hat u_j(w,-t)-\hat u(w,-t)\Bigr)\ol{\hat g(z,-t)}e^{it\hat\varphi(z,w)}d\mu(z)d\mu(w)dt}\\
&\leq \int^\infty_0t^n\abs{\hat u_j(w,-t)-\hat u(w,-t)}\abs{\ol{\hat g(z,-t)}e^{it\hat\varphi(z,w)}}d\mu(z)d\mu(w)dt\\
&\leq\int\Bigr(\int\abs{\hat u_j(w,-t)-\hat u(w,-t)}^2d\mu(w)dt\Bigr)^{\frac{1}{2}}\Bigr(\int \abs{t^n\ol{\hat g(z,-t)}e^{it\hat\varphi(z,w)}}^2d\mu(w)dt\Bigr)^{\frac{1}{2}}d\mu(z)\\
&\To 0\ \ \mbox{as $j\To\infty$}.
\end{split}
\end{equation}

From \eqref{e-gue170602II}, \eqref{e-gue170602I} and \eqref{e-gue170602}, the lemma follows.
\end{proof}

 We need

 \begin{lemma}\label{l-gue170603}
 Fix $t>0$. Let $g(z)\in C^\infty(\Complex^n)$ be any holomorphic function with $\int\abs{g(z)}^2e^{-2t\abs{\lambda}\abs{z}^2}d\mu(z)<+\infty$, where $\abs{\lambda}\abs{z}^2=\sum^n_{j=1}\abs{\lambda_j}\abs{z_j}^2$. Then,
 \[e^{-t\abs{\lambda}\abs{z}}g(z)=\frac{\abs{\lambda_1}\cdots\abs{\lambda_n}}{\pi^{n}}\int_{\mathbb C^n}t^ne^{-t\abs{\lambda}\abs{z-w}^2-t\lambda(\ol zw-z\ol w)}e^{-t\abs{\lambda}\abs{w}^2}g(w)d\mu(w),\]
 where $\lambda(\ol zw-z\ol w)=\sum^n_{j=1}\lambda_j(\ol z_jw_j-z_j\ol w_j)$.
 \end{lemma}

 \begin{proof}
 We have
 \begin{equation}\label{e-gue170603c}
 \begin{split}
 &\int_{\mathbb C^n}t^ne^{-t\abs{\lambda}\abs{z-w}^2-t\lambda(\ol zw-z\ol w)}e^{-t\abs{\lambda}\abs{w}^2}g(w)d\mu(w)\\
 &=\int_{\mathbb C^n}t^ne^{-2t\abs{\lambda}\abs{z-w}^2}\Bigr(e^{-2t\abs{\lambda}\ol zw+t\abs{\lambda}\abs{z}^2}g(w)\Bigr)d\mu(w).
 \end{split}
 \end{equation}
 From Cauchy integral formula, it is easy to see that
  \begin{equation}\label{e-gue170603cI}
 \begin{split}
 \int t^ne^{-2t\abs{\lambda}\abs{z-w}^2}h(w)d\mu(w)&=h(z)\int t^ne^{-2t\abs{\lambda}\abs{z}^2}d\mu(z)\\
 &=h(z)\frac{\pi^n}{\abs{\lambda_1}\cdots\abs{\lambda_n}},
 \end{split}
 \end{equation}
 for every holomorphic function $h(z)\in C^\infty(\Complex^n)$ with $\int\abs{h(z)}^2e^{-2t\abs{\lambda}\abs{z}^2}d\mu(z)<+\infty$. From \eqref{e-gue170603cI} and \eqref{e-gue170603c}, the lemma follows.
 \end{proof}

 We need

 \begin{lemma}\label{lemma CR}
 Let $u\in \mathcal H^0_b(H_{n+1})$, then $\hat u(z, -t)=0$ for a.e. $t\in (-\infty, 0)$.
\end{lemma}

\begin{proof}
Suppose $u\in\mathcal H^0_b(H_{n+1}).$ From \eqref{f1}, we see that
\begin{equation}
\frac{\partial}{\partial\overline z_j}\left(\hat u(z, -t)e^{t\lambda|z|^2}\right)=0, \ \ \mbox{for a.e. $t\in\mathbb R$}.
\end{equation}
Thus, $\hat u(z, -t)e^{t\lambda|z|^2}$ is a holomorphic function on $\mathbb C^n$, for a.e. $t\in\mathbb R$. From Parseval's formula, we can check that
\begin{equation}
\int_{\mathbb R}\int_{\mathbb C^n}|\hat u(z, -t)e^{t\lambda|z|^2}|^2 e^{-2t\lambda|z|^2}d\mu(z)dt=(2\pi)\int_{H_{n+1}}|u(z, x_{2n+1})|^2d\mu(z)dx_{2n+1}<\infty.
\end{equation}
It follows that
\begin{equation}
\int_{\mathbb C^n}|\hat u(z, -t)e^{t\lambda|z|^2}|^2e^{-2t\lambda|z|^2}d\mu(z)<\infty,\ \ \mbox{for a.e. $t\in\mathbb R$}.
\end{equation}
Fix $t_0\in (-\infty,0)$ so that $\int_{\mathbb C^n}|\hat u(z, -t_0)e^{t_0\lambda|z|^2}|^2e^{-2t_0\lambda|z|^2}d\mu(z)<\infty$ and $\hat u(z, -t_0)e^{t_0\lambda|z|^2}$ is a holomorphic function on $\mathbb C^n$. We write $\hat u(z,-t_0)e^{t_0\lambda\abs{z}^2}=\sum_{\alpha\in\mathbb N^n_0}c_\alpha(t_0)z^\alpha$.
	 Fix a $\alpha_0\in\mathbb N^n_0$. It is easy to see that
	\[|c_{\alpha_0}(t_0)|^2\int\abs{z^{\alpha_0}}^2e^{-2t_0\lambda\abs{z}^2}\leq \int_{\C^n}|\hat u(z,-t_0)e^{t_0\lambda\abs{z}^2}|^2e^{-2t_0\lambda|z|^2}d\mu(z)<\infty.\]
	In view of Lemma \ref{lem:monomialintegrals}, we see that $\int\abs{z^{\alpha_0}}^2e^{-2t_0\lambda\abs{z}^2}=\infty$ and hence $c_{\alpha_0}(t_0)=0$ and thus
	$\hat u(z,-t_0)=0$. The lemma follows.
\end{proof}

Now, we can prove

\begin{theorem}\label{t-gue170603}
 Let $u\in \mathcal H^0_b(H_{n+1})$. Then $\Td S^{(0)}u=u$.
\end{theorem}

\begin{proof}
Fix $g\in C^\infty_0(H_{n+1})$. We only need to prove that
\begin{equation}\label{e-gue170603cr}
(\,\Td S^{(0)}u\,|\,g\,)=(\,u\,|\,g\,).
\end{equation}
From Lemma~\ref{l-gue170602}, we have
\begin{equation}\label{e-gue170603crI}
\begin{split}
&(\Td S^{(0)}u\,|\,g\,)\\
&=\frac{\abs{\lambda_1}\cdots\abs{\lambda_n}}{2\pi^{n+1}}\int^\infty_0t^n\hat u(w,-t)\ol{\hat g(z,-t)}e^{-t\abs{\lambda}\abs{z-w}^2-t\lambda(\ol zw-z\ol w)}d\mu(z)d\mu(w)dt.\end{split}\end{equation}
From Fubini's theorem, we have
\begin{equation}\label{e-gue170603crII}
\begin{split}
&\int^\infty_0t^n\hat u(w,-t)\ol{\hat g(z,-t)}e^{-t\abs{\lambda}\abs{z-w}^2-t\lambda(\ol zw-z\ol w)}d\mu(z)d\mu(w)dt\\
&=\int^\infty_0\Bigr(\int t^n\hat u(w,-t)\ol{\hat g(z,-t)}e^{-t\abs{\lambda}\abs{z-w}^2-t\lambda(\ol zw-z\ol w)}d\mu(z)d\mu(w)\Bigr)dt\\
&=\int^\infty_0\Bigr(\int\Bigr(\int t^n\hat u(w,-t)\ol{\hat g(z,-t)}e^{-t\abs{\lambda}\abs{z-w}^2-t\lambda(\ol zw-z\ol w)}d\mu(w)\Bigr)d\mu(z)\Bigr)dt.
\end{split}
\end{equation}
From \eqref{f1} and Fubini's theorem, we see that there is a measure zero set $B\subset\Real$ such that for all $t\notin B$,  $\hat u(z, -t)e^{t\lambda|z|^2}$ is a holomorphic function on $\Complex^n$ and $\int_{\mathbb C^n}|\hat u(z, -t)e^{t\lambda|z|^2}|^2e^{-2t\lambda|z|^2}d\mu(z)<\infty$. From this observation and Lemma~\ref{l-gue170603}, we deduce that
 \begin{equation}\label{e-gue170604}
 \hat u(z,-t)=\frac{\abs{\lambda_1}\cdots\abs{\lambda_n}}{\pi^{n}}\int^\infty_0t^ne^{-t\abs{\lambda}\abs{z-w}^2-t\lambda(\ol zw-z\ol w)}\hat u(w,-t)d\mu(w),
 \end{equation}
 for every $t\notin B$. From \eqref{e-gue170604}, \eqref{e-gue170603crII}, Lemma~\ref{lemma CR} and Parseval’s formula, we get
 \[(\Td S^{(0)}u\,|\,g\,)=\frac{1}{2\pi}\int\hat u(z,-t)\ol{\hat g(z,-t)}d\mu(z)dt=(\,u\,|\,g\,).\]
 The theorem follows.
\end{proof}

\begin{proof}[Proof of Theorem \ref{main theorem 2} ]
Let $u\in L^2(H_{n+1})$. From Lemma~\ref{l-gue170528xr}, we see that $\Td S^{(0)}u\in\mathcal H^0_b(H_{n+1})$. To show that $\Td S^{(0)}=S^{(0)}$, we only need to show that $(I-\Td S^{(0)})u\perp\mathcal H^0_b(H_{n+1})$. We observe that $\Td S^{(0)}$ is self-adjoint, that is,
\begin{equation}\label{e-gue170604w}
(\,\Td S^{(0)}g\,|\,h\,)=(\,g\,|\,\Td S^{(0)}h\,),\ \ \forall g, h\in L^2(H_{n+1}).
\end{equation}
Let $f\in\mathcal H^0_b(H_{n+1})$. From Theorem~\ref{t-gue170603} and \eqref{e-gue170604w}, we have
\[(\,(I-\Td S^{(0)})u\,|\,f\,)=(\,u\,|\,f\,)-(\,\Td S^{(0)}u\,|\,f\,)=(\,u\,|\,f\,)-(\,u\,|\,\Td S^{(0)}f\,)=(\,u\,|\,f\,)-(\,u\,|\,f\,)=0.\]
The theorem follows.
\end{proof}

\section{Proof of Theorem \ref{main theorem 3} }\label{s-gue170604}

In this Section, we will prove Theorem~\ref{main theorem 3}. We only prove the case when $q=n_{-}=n_+$. For the cases $q=n_-$, $n_-\neq n_+$ and $q=n_{+}, n_-\neq n_{-}$, the arguments are similar and simpler and therefore we omit the details.

Suppose that $\lambda_1<0,\ldots,\lambda_{n_-}<0$, $\lambda_{n_-+1}>0,\ldots,\lambda_n>0$. Let $q\in\set{n_-,n_+}$, where $n_+=n-n_-$.
Let $J_{n_-}=(1,\ldots,q)$, $J_{n_+}=(q+1,\ldots,n)$. We first need

\begin{lemma}\label{lemma vanish}
Let $u=\sideset{}{'}\sum_{l(J)=q}u_J d\overline z_J\in \mathcal H^{q}_b(H_{n+1})$. If $J\notin\set{J_{n_-}, J_{n_+}}$, then $u_J=0.$
\end{lemma}
\begin{proof}
Fix $J=(j_1, j_2, \cdots, j_q)\notin\set{J_{n_-}, J_{n_+}}$ with $j_1<j_2<\cdots<j_q.$ Set
\begin{equation}
\hat\lambda|z|^2=\sum_{k\in J}\lambda_k|z_k|^2-\sum_{k\not\in J}\lambda_k|z_k|^2.
\end{equation}
Let
$F_J(z, \eta)=\hat u_J(\xi, \eta) e^{\eta\hat\lambda|z|^2}$, where $\xi_i=\overline z_i$ if $i\in J$ and $\xi_i=z_i$ if $i\not\in J.$ Then (\ref{f1}) implies that $F_J(z, \eta)$ is holomorphic, for a.e. $\eta\in\Real$.
Moreover,
\begin{equation}
\int _{\mathbb C^n}|F_J(z, \eta)|^2e^{-2\eta\hat\lambda|z|^2}\mu(z)<\infty,\ \ \mbox{for a.e. $\eta\in\mathbb R$}.
\end{equation}
From $J=(j_1, j_2, \cdots, j_q)\notin\set{J_{n_-}, J_{n_+}}$, by using  and Lemma \ref{lem:monomialintegrals}, we see that
\begin{equation}\label{e-gue170604lc}
\int e^{-2\eta\hat\lambda|z|^2}d\mu(z)=\infty,\ \ \forall\eta\in\mathbb R.
\end{equation}
From \eqref{e-gue170604lc} and repeating the argument in the proof of Lemma \ref{lemma CR}, we deduce that $F_J(z, \eta)=0$, for a.e. $\eta\in\mathbb R, z\in\mathbb C^n$. From Parseval's formula, we deduce that $u_J=0$.
\end{proof}

Put
\begin{equation}\label{e-gue170604lcI}
\tilde S^{(q)}:=\frac{\abs{\lambda_1}\cdots\abs{\lambda_n}}{2\pi^{n+1}}n!\Bigr(I^{(q)}_{\varphi_-}\circ\tau_-+I^{(q)}_{\varphi_+}\circ\tau_+\Bigr).
\end{equation}

\begin{lemma}\label{l-gue170604lc}
For every $u\in L^2_{(0,q)}(H_{n+1})$, we have $\tilde S^{(q)}u\in\mathcal H^{q}_b(H_{n+1})$.
\end{lemma}

\begin{proof}
Let $\chi\in C^\infty_0(\Real)$ be as in the beginning of Section~\ref{s-gue170528}. Fix $\varepsilon>0$. We first show that  for all $u\in\Omega^{0,q}_0(H_{n+1})$, we have
\begin{equation}\label{CR1q}
\begin{split}
I^{(q)}_{\varphi_-,\varepsilon}\circ\tau_-u:&=\frac{1}{n!}\int_{H_{n+1}}\int_0^\infty e^{it\varphi_-(x, y)}t^n(\tau_-u)(y)\chi(\varepsilon t)dtd\mu_{H_{n+1}}(y)\\
&=\frac{1}{n!}\int_{H_{n+1}}\int_0^\infty e^{it\varphi_-(x, y)}t^nu_{J_{n_-}}(y)d\ol z^{J_{n_-}}(y)\chi(\varepsilon t)dtd\mu_{H_{n+1}}(y)\in \mathcal H^{q}_b(H_{n+1}).
\end{split}
\end{equation}
Let  $u=\sideset{}{'}\sum_{l(J)=q}u_Jd\ol z^J\in\Omega^{0,q}_0(H_{n+1})$. For every $j=1, \cdots, n$, $j\notin J_{n_-}$,  we have
\begin{equation}\label{CCR2q}
\begin{split}
&\overline Z_j\left(\int_{H_{n+1}}\int_0^\infty e^{it\varphi_-(x, y)}t^nu_{J_{n_-}}(y)d\ol z^{J_{n_-}}\chi(\varepsilon t)dt d\mu_{H_{n+1}}(y)\right)\\
=&\int_{H_{n+1}}\left(\frac{\partial}{\partial\overline z_j}+i\lambda_j z_j\frac{\partial}{\partial x_{2n+1}}\right)\varphi_-(x, y)\int_0^\infty it e^{it\varphi_-(x, y)}t^nu_{J_{n_-}}(y)\chi(\varepsilon t)dt d\mu_{H_{n+1}}(y)\\
=&\int_{H_{n+1}}\left[i\abs{\lambda_j}(z_j-w_j)+i\lambda_jw_j-i\lambda_jz_j\right]
\int_0^\infty it e^{it\varphi_-(x, y)}t^nu_{J_{n_-}}(y)d\ol z^{J_{n_-}}\chi(\varepsilon t)dt d\mu_{H_{n+1}}(y)\\
=&0.
\end{split}
\end{equation}
Similarly, for every $j=1, \cdots, n$, $j\in J_{n_-}$, we have
\begin{equation}\label{CCR2qa}
\begin{split}
&Z_j\left(\int_{H_{n+1}}\int_0^\infty e^{it\varphi_-(x, y)}t^nu_{J_{n_-}}(y)d\ol z^{J_{n_-}}\chi(\varepsilon t)dt d\mu_{H_{n+1}}(y)\right)\\
=&\int_{H_{n+1}}\left(\frac{\partial}{\partial z_j}-i\lambda_j\ol z_j\frac{\partial}{\partial x_{2n+1}}\right)\varphi_-(x, y)\int_0^\infty it e^{it\varphi_-(x, y)}t^nu_{J_{n_-}}(y)\chi(\varepsilon t)dt d\mu_{H_{n+1}}(y)\\
=&\int_{H_{n+1}}\left[i\abs{\lambda_j}(\ol z_j-\ol w_j)-i\lambda_j\ol w_j+i\lambda_j\ol z_j\right]
\int_0^\infty it e^{it\varphi_-(x, y)}t^nu_{J_{n_-}}(y)d\ol z^{J_{n_-}}\chi(\varepsilon t)dt d\mu_{H_{n+1}}(y)\\
=&0.
\end{split}
\end{equation}
From \eqref{CCR2q}, \eqref{CCR2qa} and \eqref{CR1},  we get the conclusion of (\ref{CR1q}). Let $u\in\Omega^{0,q}_0(H_{n+1})$. Since $\lim_{\varepsilon\To0^+}I^{(q)}_{\varphi_-,\varepsilon}\circ\tau_-u=I^{(q)}_{\varphi_-}\circ\tau_-u$ in $\Omega^{0,q}(H_{n+1})$ topology, we deduce that $\ddbar_bI^{(q)}_{\varphi_-}\circ\tau_-u=0$ and
$(\,I^{(q)}_{\varphi_-}\circ\tau_-u\,|\,\ddbar_bv\,)=0$, for every $v\in\Omega^{0,q}_0(H_{n+1})$. By Friedrichs' lemma, we conclude that $(\,I^{(q)}_{\varphi_-}\circ\tau_-u\,|\,\ddbar_bv\,)=0$, for every $v\in{\rm Dom\,}\ddbar_b$ and hence $I^{(q)}_{\varphi_-}\circ\tau_-u\in\mathcal H^{q}_b(H_{n+1})$. Similarly, we can repeat the argument above with minor change and deduce that $I^{(q)}_{\varphi_+}\circ\tau_+u\in\mathcal H^{q}_b(H_{n+1})$, for every $u\in\Omega^{0,q}_0(H_{n+1})$.

Let $u\in L^2_{(0,q)}(H_{n+1})$ and take $u_j\in\Omega^{0,q}_0(H_{n+1})$, $j=1,2,\ldots$, $u_j\To u$ in $L^2_{(0,q)}(H_{n+1})$ as $j\To\infty$.  From Theorem~\ref{tema2}, we see that
$I^{(q)}_{\varphi_-}\circ\tau_-u_j\To I^{(q)}_{\varphi_-}\circ\tau_-u$ in $L^2_{(0,q)}(H_{n+1})$ as $j\To\infty$,  and $I^{(q)}_{\varphi_+}\circ\tau_+u_j\To I^{(q)}_{\varphi_+}\circ\tau_+u$ in $L^2_{(0,q)}(H_{n+1})$ as $j\To\infty$.  Again, by using  Friedrichs' lemma, we conclude that $\ddbar_b(I^{(q)}_{\varphi_-}\circ\tau_-+I^{(q)}_{\varphi_+}\circ\tau_+\Bigr)u=0$, $\ol{\pr}^*_b(I^{(q)}_{\varphi_-}\circ\tau_-+I^{(q)}_{\varphi_+}\circ\tau_+\Bigr)u=0$ and hence $(I^{(q)}_{\varphi_-}\circ\tau_-+I^{(q)}_{\varphi_+}\circ\tau_+\Bigr)u\in\mathcal H^{q}_b(H_{n+1})$. The lemma follows.
\end{proof}

Let
\begin{equation}\label{e-gue170605m}
\hat{T}^{1, 0}H_{n+1}:={\rm span}_{\mathbb C}\{\frac{\partial}{\partial z_j}-i\abs{\lambda_j}\overline z_j\frac{\partial}{\partial x_{2n+1}}, j=1, \cdots, n\}.
\end{equation}
Then, $\hat{T}^{1, 0}H_{n+1}$ is a CR structure of $H_{n+1}$. Let $\hat\ddbar_b$ be the tangential Cauchy Riemann operator with respect to $\hat{T}^{1, 0}H_{n+1}$ and
let $\hat S^{(0)}:L^2(H_{n+1})\To{\rm Ker\,}\hat\ddbar_b$ be the associated Szeg\H{o} projection. Put
\[\begin{split}
&\hat\varphi(x, y)\\
&=-x_{2n+1}+y_{2n+1}+i\sum_{j=1}^n|\lambda_j||z_j-w_j|^2+i\sum_{j=1}^n
\abs{\lambda_j}(\overline z_jw_j-z_j\overline w_j)\in C^\infty(H_{n+1}\times H_{n+1}).\end{split}\]
From Theorem~\ref{main theorem 2}, we see that
\begin{equation}\label{e-gue170604qa}
\hat S^{(0)}=\frac{\abs{\lambda_1}\cdots\abs{\lambda_n}}{2\pi^{n+1}}n!I^{(0)}_{\hat\varphi}\ \ \mbox{on $L^2(X)$},
\end{equation}
where $I^{(0)}_{\hat\varphi}:L^2(H_{n+1})\To L^2(H_{n+1})$ is defined as in \eqref{e-gue170525I}, \eqref{e-gue170525II}.

Let $u(x)=u_{J_{n_-}}(z,x_{2n+1})d\ol z^{J_{n_-}}+u_{J_{n_+}}(z,x_{2n+1})\in L^2_{(0,q)}(H_{n+1})$. Put
\begin{equation}\label{e-gue170605}
\begin{split}
&v_{J_{n_-}}(z,x_{2n+1}):=u_{J_{n_-}}(\ol z_1,\ldots,\ol z_{n_-},z_{n_-+1},\ldots,z_n,x_{2n+1})\in L^2(H_{n+1}),\\
&v_{J_{n_+}}(z,x_{2n+1}):=u_{J_{n_+}}(z_1,\ldots,z_{n_-},\ol z_{n_-+1},\ldots,\ol z_n,-x_{2n+1})\in L^2(H_{n+1}).
\end{split}
\end{equation}
It is straightforward to see that
\begin{equation}\label{e-gue170605I}
\begin{split}
&(I^{(q)}_{\varphi_-}\circ\tau_- u)(z,x_{2n+1})=(I^{(0)}_{\hat\varphi}v_{J_{n_-}})(\ol z_1,\ldots,\ol z_{n_-},z_{n_-+1},\ldots,z_n,x_{2n+1})d\ol z^{J_{n_-}},\\
&(I^{(q)}_{\varphi_+}\circ\tau_+ u)(z,x_{2n+1})=(I^{(0)}_{\hat\varphi}v_{J_{n_+}})(z_1,\ldots,z_{n_-},\ol z_{n_-+1},\ldots,\ol z_n,-x_{2n+1})d\ol z^{J_{n_+}}.
\end{split}
\end{equation}

Now, we can prove

\begin{theorem}\label{t-gue170604lc}
 Let $u\in \mathcal H^q_b(H_{n+1})$. Then $\Td S^{(q)}u=u$.
\end{theorem}

\begin{proof}
Let $u=\sideset{}{'}\sum_{l(J)=q}u_Jd\ol z^J\in \mathcal H^q_b(H_{n+1})$. From Lemma~\ref{lemma vanish}, we see that $u(x)=u_{J_{n_-}}(z,x_{2n+1})d\ol z^{J_{n_-}}+u_{J_{n_+}}(z,x_{2n+1})d\ol z^{J_{n_+}}$.
Let
\[v_{J_{n_-}}(z,x_{2n+1})\in L^2(H_{n+1}),\ \ v_{J_{n_+}}(z,x_{2n+1})\in L^2(H_{n+1})\]
be as in \eqref{e-gue170605}. From \eqref{CR1}, we see that
$\hat\ddbar_bv_{J_{n_-}}=0$ and $\hat\ddbar_bv_{J_{n_+}}=0$, where $\hat\ddbar_b$ is the tangential Cauchy-Riemann operator with respect to the CR stricture $\hat{T}^{1, 0}H_{n+1}$ in \eqref{e-gue170605m}. Hence, we find
\begin{equation}\label{e-gue170605mI}
\hat S^{(0)}v_{J_{n_-}}=v_{J_{n_-}},\ \ \hat S^{(0)}v_{J_{n_+}}=v_{J_{n_+}},
\end{equation}
where $\hat S^{(0)}:L^2(H_{n+1})\To{\rm Ker\,}\hat\ddbar_b$ is the Szeg\H{o} projection. From \eqref{e-gue170605mI}, \eqref{e-gue170604qa}, \eqref{e-gue170605} and \eqref{e-gue170605I}, we get $\Td S^{(q)}u=u$. The theorem follows.
\end{proof}

\begin{proof}[Proof of Theorem \ref{main theorem 3}]
Let $u\in L^2_{(0,q)}(H_{n+1})$. From Lemma~\ref{l-gue170604lc}, we see that $\Td S^{(q)}u\in\mathcal H^q_b(H_{n+1})$. To show that $\Td S^{(q)}=S^{(q)}$, we only need to show that $(I-\Td S^{(q)})u\perp\mathcal H^q_b(H_{n+1})$. We observe that $\Td S^{(q)}$ is self-adjoint, that is,
\begin{equation}\label{e-gue170604wf}
(\,\Td S^{(q)}g\,|\,h\,)=(\,g\,|\,\Td S^{(q)}h\,),\ \ \forall g, h\in L^2_{(0,q)}(H_{n+1}).
\end{equation}
Let $f\in\mathcal H^q_b(H_{n+1})$. From Theorem~\ref{t-gue170604lc} and \eqref{e-gue170604wf}, we have
\[(\,(I-\Td S^{(q)})u\,|\,f\,)=(\,u\,|\,f\,)-(\,\Td S^{(q)}u\,|\,f\,)=(\,u\,|\,f\,)-(\,u\,|\,\Td S^{(q)}f\,)=(\,u\,|\,f\,)-(\,u\,|\,f\,)=0.\]
The theorem follows.
\end{proof}
\section{Relations to weighted Bergman kernels on \(\C^n\)}\label{sec:relbergszeg}
In this section we show how the Szeg\H{o} kernel on the Heisenberg group is related to a weighted Bergman kernel on \(\C^n\) (see Theorem \ref{thm:szegobergman}). The connection mainly depends on Lemma \ref{fo1}. 
We restrict ourselves to the case  \(\lambda_1,\ldots,\lambda_n>0\) which reduces the problem to the study of the Bergman kernel for holomorphic functions. However, a generalization of the relation between Szeg\H{o}- and Bergman kernels to the case \(\lambda_1\leq\ldots\leq\lambda_q<0<\lambda_{q+1}\leq\ldots\leq\lambda_n\) is possible.

Let \(\psi\colon\C^n\rightarrow \R\) be a smooth function. We denote by \(L^2(\C^n,\psi)\) the weighted \(L^2\) space with norm
\[\|f\|_\psi^2=\int_{\C^n}|f(z)|^2e^{-2\psi(z)}d\mu(z)\]
and let \(H^0_{\psi}(\C^n)=\mathcal{O}(\C^n)\cap L^2(\C^n,\psi)\) be the space of holomorphic functions with finite weighted \(L^2\)-norm.
The Bergman kernel is a smooth function defined by
\[P_\psi(z,w)=e^{-(\psi(z)+\psi(w))}\sum_j s_j(z)\overline{s_j(w)}\]
where \(\{s_j\}\) is an ONB of \(H^0_{\psi}(\C^n)\).  
We have for example
\begin{align}\label{eq:reproducingBK}
f(z)e^{-\psi(z)}=\int_{\C^n}P_\psi(z,w)f(w)e^{-\psi(w)}d\mu(w)
\end{align}
for any \(f\in H^0_\psi(\C^n)\). 
Set \(\psi(w)=\sum_{j=1}^{n}\lambda_j|w_j|^2\) with \(\lambda_1,\ldots,\lambda_n>0\). Then we have
\begin{align}\label{Eq:BKforT}
P_{t\psi}(z,w)=1_{(0,\infty)}(t)\frac{t^n}{\pi^n}\lambda_1\cdot\ldots\cdot\lambda_ne^{-t\sum_{j=1}^n \lambda_j|w_j-z_j|^2-t\sum_{j=1}^n \lambda_j(w_j\overline{z}_j-\overline{w}_jz_j)}.
\end{align}

Now consider \(H_{n+1}=\C^n\times \R\). We define an operator \(\tilde{P}\colon L^2(H_{n+1})\rightarrow L^2(H_{n+1})\) by \(\tilde{P}(u)(x)=\hat{v}(z,x_{2n+1})\) with
\[v(z,t)=\int_{\C^n}P_{t\psi}(z,w)\check{u}(w,t)d\mu(w)\]
and \(v\mapsto \hat{v}\) (or \(u\mapsto \check{u}\)) denotes the Fourier transform in the last argument (or its inverse), using coordinates \(x=(z,x_{2n+1})\) and \(y=(w,y_{2n+1})\).
In other words we have:
\begin{definition}\label{def:segobergman}
	Given \(u\in L^2(H_{n+1})\), then
	\begin{align*}
	\tilde{P}(u)=\frac{1}{2\pi}\int_{\R}e^{-itx_{2n+1}}\left(\int_{\C^n}P_{t\psi}(z,w)\left(\int_{\R}e^{ity_{2n+1}}u(w,y_{2n+1})dy_{2n+1}\right)d\mu(w)\right)dt, 
	\end{align*}
where the order of integration is important. The integrals  are all well defined by Lemma \ref{2017-06-29a1} and  the extensions of the Fourier transform in the \(L^2(H_{n+1})\).
\end{definition}
\begin{remark}
	Note that given \(u\in C_0^\infty(H_{n+1})\) we have  \(\)
	\begin{align*}
	\tilde{P}&(u)(x)=\\
&\lim_{\varepsilon\rightarrow0}\frac{1}{2\pi}\int_{\R}\int_{H_{n+1}}
\chi(\varepsilon t)
e^{-it(x_{2n+1}-y_{2n+1})}P_{t\psi}(z,w)u(y)d\mu_{H_{n+1}}(y)dt,
	\end{align*}
	 where $\chi\in C^\infty_0(\Real)$ with $\chi(x)=1$ if $\abs{x}\leq 1$ and $\chi(x)=0$ if $\abs{x}>2$.
	  Hence we  formally write
\begin{equation}
\begin{split}
\tilde{P}(x,y)
&=\int_{\R} \frac{1}{2\pi}e^{-it(x_{2n+1}-y_{2n+1})}P_{t\psi}(z,w)dt\\ &=\frac{\lambda_1\cdots\lambda_n}{2\pi^{n+1}}\int_0^\infty e^{it\varphi(x, y)}t^ndt
\end{split}
\end{equation}
for the distribution kernel \(\tilde{P}(x,y)\) of \(\tilde{P}\), using \(x=(z,x_{2n+1})\) and \(y=(w,y_{2n+1})\). Thus, the operator $\tilde P$ is a complex Fourier integral operator in the sense of \cite{GrSj94}.
\end{remark}
We need to show that \(\tilde{P}\) is well defined in the sense that \(\tilde{P}(L^2(H_{n+1}))\subset L^2(H_{n+1})\).
\begin{lemma}\label{lem:projectionL2}
	Given \(u\in L^2(H_{n+1})\), one has \(\tilde{P}(u)\in L^2(H_{n+1})\) with \(\|\tilde{P}(u)\|\leq\|u\|\).
\end{lemma}
\begin{proof}
	We have that the Fourier transform in the last argument and its inverse preserve \(L^2(H_{n+1})\). More precisely, given \(u\in L^2(H_{n+1})\) we find \(\hat{u},\check{u}\in L^2(H_{n+1})\) with \((2\pi)^{-1}\|\hat{u}\|=\|u\|=2\pi\|\check{u}\|\). Then the proof can be deduced from the following Lemma.
\end{proof}
\begin{lemma}\label{2017-06-29a1}
	Given \(u\in L^2(H_{n+1})\), one has \(v\in L^2(H_{n+1})\) with \(\|v\|\leq\|u\|\), where
	\[v(z,t)=\int_{\C^n}P_{t\psi}(z,w)u(w,t)d\mu(w)\]
	for almost every \(t\in\R\).
\end{lemma}
\begin{proof}
	Since \(u\in L^2(H_{n+1})\) we find \(u(\cdot,t)\in L^2(\C^n)\) for all \(t\in A\) for some \(A\subset \R\), such that \(\R\setminus A\) has zero measure. We find \(u(\cdot,t)e^{t\psi(\cdot)}\in L^2(\C^n,t\psi)\) and hence it has a unique decomposition
	\[u(z,t)e^{t\psi(z)}=f(z,t)+g(z,t)\] with
	\(f(\cdot,t)\in H^0_{t\psi}(\C^n)\) and \(g(\cdot,t)\in H^0_{t\psi}(\C^n)^{\perp}\) for all \(t\in A\). We write \(u(z,t)=f(z,t)e^{-t\psi(z)}+g(z,t)e^{-t\psi(z)}\) and using the properties of the Bergman kernel we find
	\begin{align}\label{eq:represantf}
	f(z,t)e^{-t\psi(z)}=\int_{\C^n}P_{t\psi}(z,w)u(w,t)d\mu(w)
	\end{align}
	for all \(t\in A\).
	Combining (\ref{Eq:BKforT}) and (\ref{eq:represantf}) we find that
	\[(z,t)\mapsto\begin{cases}
	f(z,t)e^{-t\psi(z)}&, \text{ if } t\in A,\\
	0 &, \text{ else}.
	\end{cases}\]
	and hence \((z,t)\mapsto 1_{A}(t)g(z,t)e^{-t\psi(z)}\) define measurable functions on \(H_{n+1}\). 
	We have
	\[\int_{\C^n}f(z,t)\overline{g(z,t)}e^{-2t\psi(z)}d\mu(z)=0\] for all \(t\in A\) and hence
	\[I_0:=\int_{\R}\int_{\C^n}f(z,t)\overline{g(z,t)}e^{-2t\psi(z)}d\mu(z)dt=0.\]
	By positivity we have that
	\[I_1:=\int_{\R}\int_{\C^n}|f(z,t)|^2e^{-2t\psi(z)}d\mu(z)dt \, \text{ and } \, I_2:=\int_{\R}\int_{\C^n}|g(z,t)|^2e^{-2t\psi(z)}d\mu(z)dt\]
	exist in \([0,\infty]\).
	We then write
	\begin{align*}
	I_1+I_2&=I_1+I_0+\overline{I_0}+I_2=\int_{\R}\int_{\C^n}|f(z,t)e^{-t\psi(z)}+g(z,t)e^{-t\psi(z)}|^2d\mu(z)dt\\
	&=\int_{\R}\int_{\C^n}|u(z,t)|^2d\mu(z)dt=\int_{H_{n+1}}|u(z,t)|^2d\mu_{H_{n+1}}<\infty
	\end{align*}
	because \(u\in L^2(H_{n+1})\). Thus, we have \(I_1,I_2<\infty\).
	Setting
	\[v(z,t)=f(z,t)^{-t\psi(z)}=\int_{\C^n}P_{t\psi}(z,w)u(w,t)d\mu(w)\]
	for \(t\in A\) we have 
	\(\|v\|^2=I_1\leq \|u\|^2<\infty\) and hence \(v\in L^2(H_{n+1})\) with \(\|v\|\leq\|u\|\).
\end{proof}
\begin{lemma}\label{lem:projectionH}
	One has \(\tilde{P}(L^2(H_{n+1}))\subset \mathcal{H}^0_{b}(H_{n+1})\).
\end{lemma}
\begin{proof}
	Given \(u\in L^2(H_{n+1})\) we have \(h:=\tilde{P}(u)\in L^2(H_{n+1})\) with \(h=\lim_{\varepsilon\to 0}h_{\varepsilon}\) in \(L^2\) norm, where \(\{h_\varepsilon\}_{\varepsilon>0}\subset L^2(H_{n+1})\),
	\[h_\varepsilon(z,x_{2n+1}):=\int_{\R}\chi(\varepsilon t)g(z,t)e^{-t\psi(z)}e^{-itx_{2n+1}}dt\]
	for a.e. \((z,x_{2n+1})\in H_{n+1}\), $\chi\in C^\infty_0(\Real)$, $\chi(x)=1$ if $\abs{x}\leq 1$, $\chi(x)=0$ if $\abs{x}>2$ and \(g(\cdot,t)\) is holomorphic for a.e.~\(t\in\R\).
	A straightforward calculation shows that
	\[\left(\frac{\partial}{\partial \overline{z}_j}+i\lambda_jz_j\frac{\partial}{\partial x_{2n+1}}\right)h_\varepsilon=0\, \, \text{for} \, 1\leq j\leq n, \, \varepsilon>0\]
	holds in the sense of distributions and
	thus we conclude \(\left(\frac{\partial}{\partial \overline{z}_j}+i\lambda_jz_j\frac{\partial}{\partial x_{2n+1}}\right)h=0\) for \(1\leq j\leq n\) in the sense of distributions which shows, by Lemma \ref{lem:equivkernel}, \(h\in\mathcal{H}^0_{b}(H_{n+1})\).
\end{proof}

\begin{theorem}\label{thm:szegobergman}
	Let \(S^{(0)}\) denote the Szeg\H{o} projection for the space \(\mathcal{H}^0_b(H_{n+1})\) with distribution kernel \(S^{(0)}(x,y)\) and let \(\tilde{P}\) be the operator defined in Definition \ref{def:segobergman}. One has \(\tilde{P}=S^{(0)}\) and hence formally
	\[S^{(0)}(x,y)=\frac{1}{2\pi}\int_{\R}e^{-it(x_{2n+1}-y_{2n+1})}P_{t\psi}(z,w)dt\]
	where \(P_{t\psi}(z,w)\) denotes the weighted Bergman kernel on \(\C^n\) with respect to the weight \(\psi(z)=\sum_{j=1}^n\lambda_j|z_j|^2\), \(\lambda_1,\ldots,\lambda_n>0\), using the notation \(x=(z,x_{2n+1})\) and \(y=(w,y_{2n+1})\).
\end{theorem}
\begin{proof}
	Using Lemma \ref{fo1} and (\ref{eq:reproducingBK}) one finds \(\tilde{P}(u)=u\) for all \(u\in \mathcal{H}^0_{b}(H_{n+1})\). Moreover,  by Lemma \ref{lem:projectionH} we have that \(\tilde{P}\) satisfies \(\tilde{P}(L^2(H_{n+1}))\subset \mathcal{H}^0_{b}(H_{n+1})\). By Lemma \ref{lem:projectionL2}  we find \(\|\tilde{P}\|\leq 1\). Thus \(\tilde{P}\) is the orthogonal Projection on \(\mathcal{H}^0_{b}(H_{n+1})\) and hence we have \(\tilde{P}=S^{(0)}\).
\end{proof}
\begin{center}
{\bf Acknowledgement}
\end{center}

The authors would like  to thank the Institute for Mathematics, National University of Singapore for hospitality, a comfortable accommodation and financial support during their visits in May for the program  "Complex Geometry, Dynamical Systems and Foliation Theory". A main part of this work was done when the first and third author were visiting the Institute of Mathematics, Academia Sinica in January.
\bibliographystyle{plain}

\end{document}